\documentclass[notitlepage,a4,10.5pt]{article}

\usepackage{graphicx}

\usepackage{amsmath}%
\usepackage{amsfonts}%
\usepackage{amssymb}%
\usepackage{mathrsfs}
\usepackage{amsthm}
\usepackage{float}

\usepackage{authblk}

\usepackage[left=2.5cm,right=2.5cm,top=3cm,bottom=3cm]{geometry} 

\usepackage{xcolor}

\usepackage{enumitem}

\newcommand{\mc}{\mathcal}
\newcommand{\cp}{\times}

\newcommand{\bol}{\boldsymbol}

\newcommand{\abs}[1]{\left\lvert{#1}\right\rvert}
\newcommand{\w}{\wedge}
\newcommand{\lr}[1]{\left({#1}\right)}
\newcommand{\lrs}[1]{\left[{#1}\right]}
\newcommand{\lrc}[1]{\left\{{#1}\right\}}

\newcommand{\mf}{\mathfrak}
\newcommand{\p}{\partial}

\newcommand{\ti}[1]{\textit{#1}}
\newcommand{\tb}[1]{\textbf{#1}}

\newtheorem{mydef}{\textit{Def}}
\newtheorem{remark}{\textit{Remark}}
\newtheorem{theorem}{\textit{Theorem}}
\newtheorem{proposition}{\textit{Proposition}}
\newtheorem{lemma}{\textit{Lemma}}

\newcommand{\eq}[1]{\begin{equation}\begin{split}{#1}\end{split}\end{equation}}
\newcommand{\sys}[2]{\begin{subequations}\begin{align}{#1}\end{align}\label{#2}\end{subequations}}

\begin{document}

\title{Steady Compressible 3D Euler Flows 
\\ in   
Toroidal Volumes  without Continuous Euclidean Isometries}
\author[1]{Naoki Sato} \author[2]{Michio Yamada}
\affil[1]{National Institute for Fusion Science, \protect\\ 322-6 Oroshi-cho Toki-city, Gifu 509-5292, Japan \protect\\ Email: sato.naoki@nifs.ac.jp}
\affil[2]{Research Institute for Mathematical Sciences, \protect\\ Kyoto University, Kyoto 606-8502, Japan
\protect \\ Email: yamada@kurims.kyoto-u.ac.jp}
\date{\today}
\setcounter{Maxaffil}{0}
\renewcommand\Affilfont{\itshape\small}

    \maketitle
    \begin{abstract}
We demonstrate the existence of smooth three-dimensional vector fields where the cross product between the vector field and its curl is balanced by the gradient of a smooth function, with toroidal level sets that are not invariant under continuous Euclidean isometries. This finding indicates the existence of steady compressible Euler flows, either influenced by an external potential energy or maintained by a density source in the continuity equation, that are foliated by asymmetric nested toroidal surfaces. Our analysis suggests that the primary obstacle in resolving Grad's conjecture regarding the existence of nontrivial magnetohydrodynamic equilibria arises from the incompressibility constraint imposed on the magnetic field.
    \end{abstract}

\section{Introduction}
This paper explores the existence of a unique class of steady fluid flows that lack rotational, translational, or helical symmetry, which are crucial for the advancement of next-generation fusion reactors, particularly stellarators \cite{Hel}. Unlike other fusion designs, stellarators require inherently three-dimensional configurations that are optimized to minimize particle losses at the walls of the confining vessel, making them fundamentally asymmetric. The search for such asymmetric configurations is closely tied to the unresolved Grad's conjecture, a topic that will be further elaborated below.

Although the problem is mathematical in nature, the existence of these asymmetric solutions has significant practical implications. Currently, stability calculations are performed numerically on configurations whose existence is not theoretically guaranteed, raising concerns about the reliability of such approaches in the development of stellarators. 

In this work, we investigate compressible Euler flows that are either influenced by an external potential or sustained by a density source in the continuity equation, providing evidence for the existence of asymmetric solutions within this framework. Additionally, we explore the implications of these findings for addressing Grad's conjecture, highlighting the potential impact on understanding and resolving this longstanding problem.

Let $\Omega\subset\mathbb{R}^3$ denote a smooth bounded toroidal volume with boundary $\p\Omega$. 
The present paper is concerned with the system of equations  
\sys{
&\lr{\nabla\cp\bol{u}}\cp\bol{u}=\nabla P~~~~{\rm in}~~\Omega,\\
&\bol{u}\cdot\bol{n}=0,~~~~
P=P_b~~~~{\rm on}~~\p\Omega.
}{sys1}
Here, $\bol{u}\lr{\bol{x}}$ is the three-dimensional velocity field, $P\lr{\bol{x}}$ the (effective) pressure field, $\bol{n}$ the unit outward normal to $\p\Omega$, $P_b\in\mathbb{R}$, and $\bol{x}=\lr{x,y,z}\in\mathbb{R}^3$. Our main result is the following:

\begin{theorem}
For suitable choice of the toroidal domain $\Omega$, system \eqref{sys1} admits  solutions $\lr{\bol{u},P}\in C^{\infty}\lr{\Omega}$ such that 
$\nabla\cp\bol{u}$ and $\bol{u}$ are not everywhere collinear and level sets of $P$ are not invariant under continuous Euclidean isometries. 
\end{theorem}

As discussed in detail  in sec. 2, system \eqref{sys1} describes a steady compressible Euler flow either subject to an external potential $V\lr{\bol{x}}$ or 
sustained by a source/sink term $S\lr{\bol{x}}$ on the right-hand side of the continuity equation. Furthermore, it is a special case of the more general system of equations 
\begin{subequations}
\begin{align}
&\lr{\nabla\cp\bol{u}}\cp\bol{u}=\lambda\nabla P
,~~~~ \nabla\cdot\lr{\varrho\bol{u}}=S~~~~{\rm in}~~\Omega,\\
&\bol{u}\cdot\bol{n}=0,~~~~
P=P_b~~~~{\rm on}~~\p\Omega,
\end{align}\label{sys2}
\end{subequations}
where $\varrho\lr{\bol{x}}$ represents the fluid mass density and $\lambda\lr{\bol{x}}$ an integration factor that can be physically interpreted as a type of pressure anisotropy.  
Due to the well known correspondence between the Euler equations and magnetohydrodynamics (MHD), when $\nabla\cdot\bol{u}=0$ system \eqref{sys1} reduces to the magnetohydrostatics (MHS) equations \cite{Kruskal,Moffatt85}, 
\sys{
&\lr{\nabla\cp\bol{B}}\cp\bol{B}=\nabla P~~~~{\rm in}~~\Omega,\\
&\nabla\cdot\bol{B}=0~~~~{\rm in}~~\Omega,\\
&\bol{B}\cdot\bol{n}=0,~~~~P=P_b~~~~{\rm on}~~\p\Omega.  
}{MHS}
These equations, which can be obtained from \eqref{sys2} by setting $\bol{u}=\bol{B}$, $S=0$, $V=0$, $\lambda=1$, and $\varrho=\varrho_0\in\mathbb{R}_{\geq 0}$, describe steady force balance between the Lorentz force $\lr{\nabla\cp\bol{B}}\cp\bol{B}$ associated with the magnetic field $\bol{B}\lr{\bol{x}}$ and the pressure gradient $\nabla P$ within MHD. 
For this reason, they play an important role in the design and development of the confining magnetic field of proposed toroidal nuclear fusion reactors known as tokamaks (which are axially symmetric) and stellarators (which exhibit a 3D geometry) \cite{Hel}. 
Furthermore, they pose a fundamental theoretical question: whether a force that is generally non-conservative (the Lorentz force) can be balanced by a conservative one (the pressure gradient).

At present, a general theory concerning the existence of regular solutions of the nonlinear first order system of PDEs \eqref{MHS} for the unknowns $\lr{\bol{B},P}$ is not available \cite{LoSurdo}. 
This is because the characteristic surfaces associated with the hyperbolic part of system \eqref{MHS} depend explicitly on the unknown $\bol{B}$, a fact that makes the 
existence of global (that hold throughout $\Omega$) solutions a challenging mathematical problem \cite{Grad60,Yos90}.
Methods to obtain solutions usually rely on the elimination of the nonlinearity in the governing equations by setting $\nabla P=\bol{0}$, or on the assumption of a symmetry (invariance under a continuous Euclidean isometry) in the candidate solution $\lr{\bol{B},P}$ that enables the removal of the hyperbolic part of the system, leaving an elliptic PDE (the Grad-Shafranov equation \cite{Grad58}). 

The list below summarizes 
known results concerning the existence and regularity of solutions of the magnetohydrostatic equations \eqref{MHS}, and the parent systems \eqref{sys1} and \eqref{sys2}. Here, $\Omega$ is the smooth bounded toroidal domain introduced above, unless differently specified. 
It should be noted that the  domain $\Omega$ is chosen to be a toroidal volume not only because of nuclear fusion applications, but also due to the hairy ball and Poincar\'e-Hopf theorems  
\cite{Eisen}, and the closely related Arnold's structure theorem \cite{Arnold,SalasI}, 
which restrict the topologies of the domain $\Omega$ compatible with system \eqref{MHS}.
\begin{enumerate}[label=\arabic*., listparindent=1.5em]
\item \ti{Vacuum solutions $\nabla\cp\bol{B}=\bol{0}$ of system \eqref{MHS}}. In this case the magnetic field is the unique Neumann harmonic vector field in the toroidal domain $\Omega$ \cite{Schwarz}. 
\item \ti{Beltrami field solutions $\nabla\cp\bol{B}=\alpha\bol{B}$ of system \eqref{MHS}}. In this case magnetic field and electric current are collinear, with proportionality factor $\alpha\lr{\bol{x}}$. Existence of solutions depends on whether the proportionality factor is a real constant \cite{YosGiga} or not \cite{SalasBel}. 
In the former case, solutions in the Sobolev space $H^1\lr{\Omega}$ exist, with the admissible constants $\alpha$ depending on the topology of the domain under consideration. On the other hand, 
nonconstant proportionality factors introduce obstructions to the existence of solutions \cite{SalasBel,SatoYamadaBel, Clelland, Abe,SalasBel2}. An open question is whether the function $\alpha$, which corresponds to a first integral $\bol{B}\cdot\nabla\alpha=0$ of the magnetic field, can be chosen to be constant on $\p\Omega$ for a general  toroidal domain $\Omega$. 
\item \ti{Symmetric solutions of system \eqref{MHS}}. In this case one assumes that there exist constants $\bol{a},\bol{b}\in\mathbb{R}^3$ 
and a vector field $\bol{\xi}_{E}=\bol{a}+\bol{b}\cp\bol{x}\neq\bol{0}$   
such that the Lie-derivatives $\mf{L}_{\bol{\xi}_E}\bol{B}=\bol{0}$ and  $\mf{L}_{\bol{\xi}_E}P=0$ identically vanish in $\Omega$. The vector field $\bol{\xi}_E$ is the generator of a continuous Euclidean isometry, implying that the solution $\lr{\bol{B},P}$ is invariant under a combination of translations and rotations. Under this assumption, system \eqref{MHS} reduces to the Grad-Shafranov equation, a nonlinear second order elliptic PDE for the flux function \cite{Grad58, Eden1, Eden2}. The existence of solutions can be discussed within the framework of standard elliptic theory.   
\item \ti{Stepped pressure solutions of system \eqref{MHS}}. In this setting, the toroidal domain $\Omega$ is partitioned into sub-regions   
filled by Beltrami fields. 
At the boundaries of these sub-regions the value of the pressure jumps, realizing a pressure gradient at the price of low regularity of solutions \cite{Bruno,Bob15,SalasMHS}.  

\item \ti{Nonlinear asymmetric solutions of system \eqref{MHS}}.  Here, the toroidal domain $\Omega$ is not invariant under any  continuous Euclidean isometry, and   equilibria are supported by a non-vanishing pressure gradient. According
to Grad's conjecture \cite{Grad}, regular solutions of this type  
do not exist. 
Given the absence of a formal statement, there is no accepted standard for the degree of regularity implied in the conjecture. However, given that $\bol{B}$,  $\nabla\cp\bol{B}$, and $\nabla P$ are physical observables, it seems reasonable to demand at least $\lr{\bol{B},P}\in C^1\lr{\Omega}$. 
In this study, we will adopt the  formulation of Grad's conjecture for smooth solutions:

\ti{`
Given a smooth, bounded toroidal domain \(\Omega \subset \mathbb{R}^3\) that is not invariant under any continuous Euclidean isometry, there do not exist smooth solutions to the system \eqref{MHS} where the vector fields \(\bol{B}\) and \(\nabla \times \bol{B}\) are non-collinear.'} 

It has been shown \cite{Pasq} that the conjecture does not hold 
for $\bol{B}\in H^1\lr{\Omega}$. Indeed, nontrivial solutions in this class have been obtained  through a Voigt regularization  scheme \cite{Cao06,Larios14,Larios10,Ramos10,Levant10} for the MHD equations in which regularizing differential operators are introduced, enabling 
convergence of time-dependent solutions to nontrivial  equilibria \eqref{MHS}. An open question is whether this type of regularization can be applied to improve regularity beyond $H^1\lr{\Omega}$. 
It is also known \cite{Davids3} that the conjecture breaks down 
in non-Euclidean metrics: 
there exist MHD
equilibria with non-vanishing  pressure gradients on compact three-manifolds with or
without boundary 
and with no continuous Killing symmetries.

\item \ti{Nonlinear asymmetric solutions of system \eqref{sys2} with $\nabla\cdot\bol{u}=0$, $\varrho\in\mathbb{R}_{\geq0}$, and  $S=0$.}  
These solutions include MHD equilibria sustained by   anisotropic pressure \cite{Grad67}. 
Given a well-behaved set of 
nested flux surfaces $P$, it has been shown \cite{SatoYamada23} that solutions $\bol{u}\in C^{\infty}\lr{\Omega}$, with $\Omega$ a smooth bounded \ti{hollow} toroidal volume, exist for a suitable choice of the integration factor $\lambda\lr{\bol{x}}$.
This result suggests that one of the key obstructions in solving system \eqref{MHS} is represented by the conservative nature of the pressure gradient  
(the exactness of the $1$-form $dP$): force balance can be achieved as soon as this property is violated. 


\item \ti{Nonlinear asymmetric solutions of system \eqref{sys1}}. These are the compressible solutions examined in the present work, and fall in the class $\lr{\bol{u},P}\in C^{\infty}\lr{\Omega}$ (see theorem 1 above). This result points to the incompressible nature of the magnetic field as the other key obstruction in solving system \eqref{MHS}. 
This observation is consistent with the fact that there exist a correspondence between a first integral of a divergence-free field in three dimensions, and the existence of a continuous (not necessarily Euclidean) symmetry \cite{Davids, Davids2,Constantin21}. 
\end{enumerate}

The primary objective of this study is to prove Theorem 1. The paper is organized as follows: in Section 2, we provide a straightforward derivation of system \eqref{sys1} within the framework of the compressible Euler equations. We then perform a Clebsch reduction of system \eqref{sys1} by representing the vector field 
$\bol{u}$ using Clebsch potentials \cite{YosClebsch,YosMor}. The details of this reduction are thoroughly explained in the appendix. Section 3 presents a constructive proof of Theorem 1 by deriving explicit solutions of the reduced system. The properties of these solutions are discussed in Section 4, and Section 5 explores the implications of this study for the MHS equations \eqref{MHS}. Finally, concluding remarks are provided in Section 6.


\section{Steady compressible Euler flows with external potential or density source}
The aim of this section is to show that system \eqref{sys1} arises as a steady state in the context of the compressible Euler equations subject to an external potential in the momentum equation  or endowed with a density source in the continuity equation for mass conservation. 

We begin by considering the compressible Euler equations with external potential and density source  
in $\Omega\times [0,+\infty)$. 
Let $\bol{u}\lr{\bol{x},t}$, $\varrho\lr{\bol{x},t}$, $P'=P'\lr{\varrho}$,  $S\lr{\bol{x},t}$, and $V\lr{\bol{x},t}$ 
denote the velocity field, the mass density, the barotropic pressure, the density source, and the potential energy 
respectively. 
The equations read as 
\begin{subequations}
\begin{align}
&\varrho\frac{\p\bol{u}}{\p t}=-\varrho\bol{u}\cdot\nabla\bol{u}-\nabla P'-\rho\nabla V~~~~{\rm in}~~\Omega,
\\
&\frac{\p\varrho}{\p t}=-\nabla\cdot\lr{\varrho\bol{u}}+S~~~~{\rm in}~~\Omega,\\
&\bol{u}\cdot\bol{n}=0,~~~~
\frac{\bol{u}^2}{2}+h+V=-P_b
~~~~{\rm on}~~\p\Omega, 
\end{align}\label{Euler}
\end{subequations}
where $h\lr{\varrho}$ denotes a function of the density $\varrho$ such that $\nabla P'=\varrho\nabla h$. 
Observe that in system \eqref{Euler} the evolution equations for the internal energy density and the source $S$, as well as initial conditions are omitted since they are not needed in the  analysis of system \eqref{sys1}. 

First, let us examine the case in which there is no source term in the continuity equation, $S=0$, and the system is subject to an external potential $V$. 
Setting $-P=\bol{u}^2/2+h+V$, system 
\eqref{Euler} becomes
\begin{subequations}
\begin{align}
&\frac{\p\bol{u}}{\p t}=\bol{u}\cp\lr{\nabla\cp\bol{u}}+\nabla P~~~~{\rm in}~~\Omega,
\\
&\frac{\p\varrho}{\p t}=-\nabla\cdot\lr{\varrho\bol{u}}~~~~{\rm in}~~\Omega,\\
&\bol{u}\cdot\bol{n}=0,~~~~
P=P_b
~~~~{\rm on}~~\p\Omega.
\end{align}\label{redEuler0}
\end{subequations}
Eliminating time derivatives gives  
\begin{subequations}
\begin{align}
&\lr{\nabla\cp\bol{u}}\cp\bol{u}=\nabla P~~~~{\rm in}~~\Omega,\label{Euler30a}
\\
&\nabla\cdot\lr{\varrho\bol{u}}=0~~~~{\rm in}~~\Omega,\label{S0}\\
&\bol{u}\cdot\bol{n}=0,~~~~
P=P_b
~~~~{\rm on}~~\p\Omega.\label{Euler30c}
\end{align}\label{Euler30}
\end{subequations}
Now we see that for given $\bol{u}\lr{\bol{x}}$ and $P\lr{\bol{x}}$ solving \eqref{Euler30a} with boundary conditions \eqref{Euler30c}, the external potential $V\lr{\bol{x}}$ can be obtained from $V=-P-\bol{u}^2/2-h$, where the function $h\lr{\varrho}$ is assumed given and $\varrho\lr{\bol{x}}$ solves \eqref{S0} (provided that such solution $\varrho\lr{\bol{x}}$ exists for given $\bol{u}\lr{\bol{x}}$ under some appropriate boundary conditions). 
System \eqref{Euler30} can therefore be regarded as a system of $3$ equations for the $4$ unknowns $\lr{\bol{u},P}$. The resulting governing equations are given by system \eqref{sys1}. 

The same governing equations \eqref{sys1} can be obtained from system \eqref{Euler} without the external potential $V$ if a source term $S$ is included in the continuity equation. 
Indeed, setting $V=0$ and $-P=\bol{u}^2/2+h$, 
system \eqref{Euler} becomes
\begin{subequations}
\begin{align}
&\frac{\p\bol{u}}{\p t}=\bol{u}\cp\lr{\nabla\cp\bol{u}}+\nabla P~~~~{\rm in}~~\Omega,
\\
&\frac{\p\varrho}{\p t}=-\nabla\cdot\lr{\varrho\bol{u}}+S~~~~{\rm in}~~\Omega,\\
&\bol{u}\cdot\bol{n}=0,~~~~
P=P_b
~~~~{\rm on}~~\p\Omega.
\end{align}\label{redEuler}
\end{subequations}
Dropping time derivatives  
leads to 
\begin{subequations}
\begin{align}
&\lr{\nabla\cp\bol{u}}\cp\bol{u}=\nabla P~~~~{\rm in}~~\Omega,\label{Euler3a}
\\
&\nabla\cdot\lr{\varrho\bol{u}}=S~~~~{\rm in}~~\Omega,\label{S}\\
&\bol{u}\cdot\bol{n}=0,~~~~
P=P_b
~~~~{\rm on}~~\p\Omega.\label{Euler3c}
\end{align}\label{Euler3}
\end{subequations}
Now observe that in this system for given $\bol{u}\lr{\bol{x}}$ and $P\lr{\bol{x}}$ solving \eqref{Euler3a} with boundary conditions \eqref{Euler3c}, the density $\varrho\lr{\bol{x}}$ can be obtained from the relation $h\lr{\varrho}=-P-\bol{u}^2/2$. Hence, the density source $S\lr{\bol{x}}$ is completely determined by equation \eqref{S}. System \eqref{Euler3} can therefore be regarded as a system of $3$ equations for the $4$ unknowns $\lr{\bol{u},P}$. The resulting governing equations are given by system \eqref{sys1}. 

An alternative path to arrive at system \eqref{sys1} is to assume that the mass density is a positive constant, $\varrho=\varrho_0\in\mathbb{R}_{>0}$, set $V=0$,  and treat $P'\lr{\bol{x}}$ as an unknown. In this case, the governing equations are
\begin{subequations}
\begin{align}
&\frac{\p\bol{u}}{\p t}=-\bol{u}\cdot\nabla\bol{u}-\frac{1}{\varrho_0}\nabla P'~~~~{\rm in}~~\Omega,
\\
&\nabla\cdot\bol{u}=\frac{S}{\varrho_0}~~~~{\rm in}~~\Omega,\\
&\bol{u}\cdot\bol{n}=0,~~~~
\frac{P'}{\varrho_0}+\frac{\bol{u}^2}{2}=-P_b
~~~~{\rm on}~~\p\Omega,
\end{align}\label{Euler32}
\end{subequations}
Introducing the variable $P=-P'/\varrho_0-\bol{u}^2/2$, 
steady solutions of equation \eqref{Euler32} are now described by 
\sys{
&\lr{\nabla\cp\bol{u}}\cp\bol{u}=\nabla P~~~~{\rm in}~~\p\Omega,\\
&\nabla\cdot\bol{u}=\frac{S}{\varrho_0}~~~~{\rm in}~~\Omega,\\
&\bol{u}\cdot\bol{n}=0,~~~~P=P_b~~~~{\rm on}~~\p\Omega,
}{sysEuler2}
which leads to \eqref{sys1} once the source term $S=\varrho_0\nabla\cdot\bol{u}$ is eliminated from the system. 

It is important to emphasize that there is no corresponding magnetohydrostatic configuration associated with the steady flows described by system \eqref{sys1}. This is because Maxwell's equation \(\nabla\cdot\bol{B}=0\) prohibits source terms on the right-hand side, as magnetic monopoles are not predicted by electromagnetic theory. Moreover, even within the fluid dynamics framework of system \eqref{sys1}, there is no guarantee that a solution $\lr{\bol{u},P}$, with an external potential \(V = -P - \frac{\bol{u}^2}{2} - h\) or a source term \(S = \nabla\cdot(\varrho\bol{u})\), can be physically realized. Therefore, the solutions of system \eqref{sys1} obtained in this study should be interpreted within this theoretical context.


\section{Constructive proof of theorem 1}
System \eqref{sys1} can be reduced to a system of equations \eqref{sys7} for a set of Clebsch potentials \(\{\alpha, \Psi, \Theta, \zeta, \Phi\}\) and the pressure $P$ using the Clebsch parametrization 
\eq{\bol{u} = \nabla\Phi + \Psi\nabla\Theta + \alpha\nabla\zeta,\label{clu}} 
of the velocity field. The technical details of this reduction are provided in Appendix A. In this context, $\lr{\Psi, \Theta, \zeta}$ are treated as a coordinate system, with tangent vectors $\lr{\p_{\Psi},\p_{\Theta},\p_{\zeta}}$ and the corresponding metric tensor components denoted by \(g_{\Psi\Psi}\), \(g_{\Psi\Theta}\), \(g_{\Psi\zeta}\), \(g_{\Theta\Theta}\), \(g_{\Theta\zeta}\), and \(g_{\zeta\zeta}\) (see Eq. \eqref{mt} in Appendix A for their definitions).


If a solution of the reduced system \eqref{sys7} exists, the vector field $\bol{u}$ will be tangent to level sets of $P$. 
Assuming that \(\nabla P \neq \bol{0}\) in \(\bar{\Omega}\), 
we can use \(P\) as a coordinate and set \(\Psi = P\).
Under these assumptions, the surviving equations in system \eqref{sys7} are
\sys{
&\alpha_{\Theta}=0~~~~{\rm in}~~\Omega,\\
&\frac{\alpha+\Phi_{\zeta}+g_{\Theta \zeta}}{g_{\zeta\zeta}-\alpha_{\Psi}g_{\Theta \zeta}}\lr{g_{\Psi \zeta}-\alpha_{\Psi}g_{\Psi\Theta}}=g_{\Psi\Theta}+\Phi_{\Psi}~~~~{\rm in}~~\Omega,\\
&\frac{\alpha+\Phi_{\zeta}+g_{\Theta \zeta}}{g_{\zeta\zeta}-\alpha_{\Psi}g_{\Theta \zeta}}\lr{g_{\Theta \zeta}-\alpha_{\Psi}g_{\Theta\Theta}}=\Phi_{\Theta}+\Psi+g_{\Theta\Theta}~~~~{\rm in}~~\Omega,\\
&\Psi=\Psi_b
~~~~{\rm on}~~{\p\Omega}.
}{sys8}
Here, a lower index applied to a function denotes partial differentiation, e.g. $\alpha_{\Theta}=\p\alpha/\p\Theta$. 
We remark that in this system the unknowns are $\alpha$, $\Psi$, $\Theta$, $\zeta$, and $\Phi$, with the boundary condition that the toroidal surface $\p\Omega$ corresponds to a level set of the coordinate $\Psi$. 

Since there are $3$ equations for $5$ unknowns, we expect enough freedom to fix some of the unknowns. 
However, for the time being we shall leave such freedom and consider all quantities as variables of the problem to facilitate its solution. 
Furthermore, below we demand all quantities to be integrable or differentiable as required by the governing equations. 
Then, a possible way of solving system \eqref{sys8} consists in finding $\alpha$, $\Psi$,  $\Theta$, $\zeta$, and $\Phi$ such that in $\Omega$ the following relationships hold   
\sys{
&\alpha_{\Theta}=0,\\
&\alpha+\Phi_{\zeta}+g_{\Theta\zeta}=0,\label{alpha}\\
&\Phi_{\Psi}+g_{\Psi\Theta}=0,\\
&\Phi_{\Theta}+\Psi+g_{\Theta\Theta}=0.
}{sys9}
Suppose that $\Theta\in[a,b)$ for some $a,b\in\mathbb{R}$. 
Since $\Theta$ only appears through its spatial derivatives in $\bol{u}$, we can redefine $\Theta$ by subtracting a constant so that $\Theta\in[0,c)$ with $c=b-a$. 
Then, the last equation in \eqref{sys9} gives
\eq{
\Phi=-\Psi\Theta-\int_0^{\Theta}g_{\Theta\Theta}d\Theta'+\chi\lr{\Psi,\zeta},\label{Phi}
}
where $\chi\lr{\Psi,\zeta}=\Phi\lr{\Psi,0,\zeta}$. 
Using this expression, we find
\sys{
&\Phi_{\Psi}=-\Theta-\int_0^{\Theta}\frac{\p g_{\Theta\Theta}}{\p\Psi}d\Theta'+\chi_{\Psi},\\
&\Phi_{\Theta\zeta}=-\frac{\p g_{\Theta\Theta}}{\p\zeta}. 
}{sys10}
On the other hand, since $\alpha$ is an arbitrary function of $\Psi$ and $\zeta$, the second equation in system \eqref{sys9} implies that
\eq{
\Phi_{\Theta\zeta}=-\frac{\p g_{\Theta\zeta}}{\p\Theta}.
}
Hence, system \eqref{sys8} can be reduced to the following two coupled equations for $\Theta$ and $\zeta$, 
\sys{
&\Theta+\int_0^{\Theta}\frac{\p g_{\Theta\Theta}}{\p\Psi}d\Theta'-g_{\Psi\Theta}=\chi_{\Psi},\label{chipsi}\\
&\frac{\p g_{\Theta\Theta}}{\p\zeta}=\frac{\p g_{\Theta\zeta}}{\p\Theta}.
}{sys11}
Recalling that $\chi_{\Psi}$ is an arbitrary function of $\Psi$ and $\zeta$, this system is equivalent to 
\sys{
&\frac{\p g_{\Psi\Theta}}{\p\Theta}=1+\frac{\p g_{\Theta\Theta}}{\p\Psi},\\
&\frac{\p g_{\Theta\Theta}}{\p\zeta}=\frac{\p g_{\Theta\zeta}}{\p\Theta}.
}
{sys12}
Note that if one can find coordinates $\Psi$,  $\Theta$, and $\zeta$ satisfying \eqref{sys12}, the function $\chi_{\Psi}$ can be obtained from equation \eqref{chipsi}, 
which in turns gives $\Phi$ (eq. \eqref{Phi}) and $\alpha$ (eq. \eqref{alpha}) in consistency with system \eqref{sys9}. The functions $\alpha$, $\Psi$, $\Theta$, $\zeta$, and $\Phi$ then produce a solution $\bol{u}$ of the original problem \eqref{sys1} in the Clebsch  form \eqref{clu}. 
Of course, in this construction the coordinate $\Psi$ must be such that one of its contours corresponds to the boundary $\p\Omega$ to be consistent with boundary conditions. 

Recalling the definition of tangent vectors and their relationship with the metric coefficients, system \eqref{sys12} can be further transformed into a system for the map $\bol{x}\lr{\Psi,\Theta,\zeta}$ according to
\sys{
&\frac{\p\bol{x}}{\p\Psi}\cdot\frac{\p^2\bol{x}}{\p\Theta^2}=1+\frac{\p\bol{x}}{\p\Theta}\cdot\frac{\p^2\bol{x}}{\p\Psi\p\Theta},\\
&\frac{\p\bol{x}}{\p\Theta}\cdot\frac{\p^2\bol{x}}{\p\Theta\p\zeta}=\frac{\p\bol{x}}{\p\zeta}\cdot\frac{\p^2\bol{x}}{\p\Theta^2}.
}{sys13}
In terms of Christoffel symbols $\Gamma_{ijk}=\frac{1}{2}\lr{\frac{\p g_{ij}}{\p x^k}+\frac{\p g_{ik}}{\p x^j}-\frac{\p g_{jk}}{\p x^i}}$ 
for the metric tensor $g_{ij}$ with  coordinates $x^1=\Psi$, $x^2=\Theta$, and $x^3=\zeta$, system \eqref{sys13} can be shown to be equivalent to
\sys{
&\Gamma_{\Psi\Theta\Theta}=1+\Gamma_{\Theta\Psi\Theta},\\&\Gamma_{\Theta\zeta\Theta}=\Gamma_{\zeta\Theta\Theta},
}{Chris1}
where we defined $\Gamma_{\Psi\Theta\Theta}=\Gamma_{122}$, $\Gamma_{\Theta\Psi\Theta}=\Gamma_{212}$, $\Gamma_{\Theta\zeta\Theta}=\Gamma_{232}$, and $\Gamma_{\zeta\Theta\Theta}=\Gamma_{322}$. 
The above facts can be summarized in the following statement:
\begin{proposition}
Let $\lr{\Psi,\Theta,\zeta}$ be a  coordinate system in a smooth bounded toroidal domain $\Omega\subset\mathbb{R}^3$ satisfying  \eqref{sys13}. Suppose that the boundary $\p\Omega$ corresponds to a contour of the function $\Psi$. 
Take $\Phi$,  $\alpha$, and $\chi_{\Psi}$ as defined by equations \eqref{Phi},  \eqref{alpha}, and \eqref{chipsi}. Then, the vector field $\bol{u}=\nabla\Phi+\Psi\nabla\Theta+\alpha\nabla\zeta$ solves \eqref{sys1} with $P=\Psi$. Furthermore, 
\eq{\bol{u}=-g_{\Psi\Theta}\nabla\Psi-g_{\Theta\Theta}\nabla\Theta-g_{\Theta\zeta}\nabla\zeta=-\p_{\Theta}.}   
\end{proposition}
\begin{proof}
The proof of this statement directly follows from the identities satisfied by the coordinates $\lr{\Psi,\Theta,\zeta}$ and the functions $\Phi$, $\alpha$, and $\chi_{\Psi}$. 
\end{proof}

Let $\lr{r,\varphi,z}$ denote cylindrical coordinates. 
Consider the axially symmetric toroidal surfaces corresponding to level sets of the function $\Psi=\Psi_0-\lr{r-r_0}^2-z^2$ with $\Psi_0,r_0\in\mathbb{R}_{>0}$.  
One can verify that system \eqref{sys13} admits the solution
\sys{
&x=\lr{r_0+\sqrt{\Psi_0-\Psi}\cos{\Theta}}\cos{\zeta},\\
&y=\lr{r_0+\sqrt{\Psi_0-\Psi}\cos{\Theta}}\sin{\zeta},\\
&z=\sqrt{\Psi_0-\Psi}\sin{\Theta},}
{sys14}
with corresponding inverse transformation 
\sys{
&\Psi=\Psi_0-\lr{r-r_0}^2-z^2,\\
&\Theta=\arctan\lr{\frac{z}{r-r_0}},\\
&\zeta=\varphi.
}{sys15}
Here, the coordinate $\Theta$ is the poloidal angle. 
This result implies that one can solve \eqref{sys1} in terms of a vector field  $\bol{u}$ and a pressure field  $P=\Psi$ having nested axially symmetric level sets. 
Let $\epsilon>0$ be a real constant. 
In order to break the axial symmetry of the function $\Psi$, we note that the perturbed transformation
\sys{
&x=\lr{r_0+\sqrt{\Psi_0-\Psi}\cos\Theta}\cos\zeta+\epsilon\delta x\lr{\zeta},\\
&y=\lr{r_0+\sqrt{\Psi_0-\Psi}\cos\Theta}\sin\zeta+\epsilon\delta y\lr{\zeta}\\
&z=\sqrt{\Psi_0-\Psi}\sin\Theta,
}
{sys15b}
remains compatible with system \eqref{sys13} provided that the perturabtions $\delta x\lr{\zeta}$ and $\delta y\lr{\zeta}$ satisfy
\eq{
\delta x_{\zeta}=f\lr{\zeta}\sin\zeta ,~~~~\delta y_{\zeta}=-f\lr{\zeta}\cos\zeta,\label{delta}
}
for some $f\lr{\zeta}$. 
For example, setting $f\lr{\zeta}=\sin^2\zeta$, we find 
\sys{
&x=\lr{r_0+\sqrt{\Psi_0-\Psi}\cos\Theta}\cos\zeta-\epsilon\cos\zeta\lr{1-\frac{1}{3}\cos^2\zeta},\\
&y=\lr{r_0+\sqrt{\Psi_0-\Psi}\cos\Theta}\sin\zeta-\frac{1}{3}\epsilon\sin^3\zeta\\
&z=\sqrt{\Psi_0-\Psi}\sin\Theta.
}{sys16}
One can verify that the transformation \eqref{sys16} satisfies
\sys{
&\frac{\p\bol{x}}{\p\Psi}=\lrc{-\frac{\cos\zeta\cos\Theta}{2\sqrt{\Psi_0-\Psi}},-\frac{\sin\zeta\cos\Theta}{2\sqrt{\Psi_0-\Psi}},-\frac{\sin\Theta}{2\sqrt{\Psi_0-\Psi}}},\\
&\frac{\p\bol{x}}{\p\zeta}=\lrc{-\sin\zeta\lr{r_0-\epsilon\sin^2\zeta+\sqrt{\Psi_0-\Psi}\cos\Theta},\cos\zeta\lr{r_0-\epsilon\sin^2\zeta+\sqrt{\Psi_0-\Psi}\cos\Theta},0}\\
&\frac{\p\bol{x}}{\p\Theta}=\lrc{-\cos\zeta\sin\Theta\sqrt{\Psi_0-\Psi},-\sin\zeta\sin\Theta\sqrt{\Psi_0-\Psi},\cos\Theta\sqrt{\Psi_0-\Psi}}\label{xtheta}\\
&\frac{\p^2\bol{x}}{\p\Psi\p\Theta}=\lrc{\frac{\cos\zeta\sin\Theta}{2\sqrt{\Psi_0-\Psi}},\frac{\sin\zeta\sin\Theta}{2\sqrt{\Psi_0-\Psi}},-\frac{\cos\Theta}{2\sqrt{\Psi_0-\Psi}}}\\
&\frac{\p^2\bol{x}}{\p\zeta\p\Theta}=\lrc{\sin\zeta\sin\Theta\sqrt{\Psi_0-\Psi},-\cos\zeta\sin\Theta\sqrt{\Psi_0-\Psi},0}\\
&\frac{\p^2\bol{x}}{\p\Theta^2}=\lrc{-\cos\zeta\cos\Theta\sqrt{\Psi_0-\Psi},-\sin\zeta\cos\Theta\sqrt{\Psi_0-\Psi},-\sin\Theta\sqrt{\Psi_0-\Psi}}.
}
{sys16a}
These expressions can be used to show that system \eqref{sys13} is satisfied  by the transformation \eqref{sys16}.  


\begin{proposition}
Suppose that $\Psi\in(-\infty,\Psi_0)$, 
$\Theta\in [0,2\pi)$, and $\zeta\in [0,2\pi)$ and set $\bol{y}=\lr{\Psi,\Theta,\zeta}$.  Define 
\eq{M=\lrc{
\bol{y}\in(-\infty,\Psi_0)\times[0,2\pi)^2:r_0+\sqrt{\Psi_0-\Psi}\cos\Theta-\epsilon\sin^2\zeta> 0}.} 
Then, 
the coordinate transformation $T\lr{\bol{y}}:M\rightarrow N\subset\mathbb{R}^3$ defined by equation \eqref{sys16} is smooth. 
Furthermore, it is  invertible with smooth inverse, i.e. for every $\bol{y}\in M$ there exist a neighborhood $U\subset M$ of $\bol{y}$ and a smooth  inverse transformation $T^{-1}\lr{\bol{x}}:V\rightarrow U$, with $V\subset N$ the image of $U$ under the map $T$, such that $T^{-1}\lr{T\lr{\bol{y}}}=\bol{y}$. 
In addition, the image $N$ satisfies 
\eq{
N'=
\lrc{\bol{x}\in\mathbb{R}^3:r>\epsilon\lr{1+\sqrt{2}}}\subset N.\label{N'}
}
\end{proposition}

\begin{proof}
The smoothness of the coordinate transformation on $M$ is immediate. 
Invertibility follows by the inverse function theorem: it suffices to verify that the Jacobian determinant of the transformation does not vanish within $M$. The Jacobian matrix $\p_{\bol{y}}\bol{x}$ can be written as
\begin{equation}
\p_{\bol{y}}\bol{x}=\begin{bmatrix}
-\frac{\cos\zeta\cos\Theta}{2\sqrt{\Psi_0-\Psi}}&-\sqrt{\Psi_0-\Psi}\cos\zeta\sin\Theta&-\sin\zeta\lr{r_0+\sqrt{\Psi_0-\Psi}\cos\Theta-\epsilon\sin^2\zeta}\\
-\frac{\sin\zeta\cos\Theta}{2\sqrt{\Psi_0-\Psi}}&-\sqrt{\Psi_0-\Psi}\sin\zeta\sin\Theta&\cos\zeta\lr{r_0+\sqrt{\Psi_0-\Psi}\cos\Theta-\epsilon\sin^2\zeta}\\
-\frac{\sin\Theta}{2\sqrt{\Psi_0-\Psi}}&\sqrt{\Psi_0-\Psi}\cos\Theta&0
\end{bmatrix}.
\end{equation}
Hence, in $M$, the Jacobian determinant is
\begin{equation}
J^{-1}=\det\lr{\p_{\bol{y}}\bol{x}}=\frac{1}{2}\lr{r_0+\sqrt{\Psi_0-\Psi}\cos\Theta-\epsilon\sin^2\zeta}=\frac{\rho}{2}-\frac{1}{2}\epsilon\sin^2\zeta,\label{Jac}
\end{equation}
where we defined the perturbed polar radius 
\eq{
\rho=\sqrt{\lr{x-\epsilon\delta x}^2+\lr{y-\epsilon\delta y}^2}=\abs{r_0+\sqrt{\Psi_0-\Psi}\cos\Theta}=r_0+\sqrt{\Psi_0-\Psi}\cos\Theta.
}
Since by hypothesis $M$ does not include the set $\Sigma=\lrc{\bol{y}\in(-\infty,\Psi_0)\times[0,2\pi)^2:r_0+\sqrt{\Psi_0-\Psi}\cos\Theta\leq\epsilon\sin^2\zeta}$, the Jacobian determinant does not vanish in $M$. 
The remaining step is to characterize the image $N$ of the map $T$ as claimed in equation \eqref{N'}. 
First, observe that by choosing $\cos\Theta\cos\zeta\neq 0$, the absolute value of coordinate $x$ in \eqref{sys16}  can be made arbitrary large by taking the limit $\Psi\rightarrow -\infty$. 
Note that this limit can be taken within $M$, i.e. by avoiding the set $\Sigma$.
A similar reasoning applies to $y$ and $z$. 
This suggests that the image of $\bol{x}\lr{\Psi,\zeta,\Theta}$ is $\mathbb{R}^3$, 
minus the loci where the Jacobian determinant $J$ becomes singular. 
When $\epsilon=0$, we have $\rho=r$ so that the singular set is the vertical axis $T\lr{\Sigma}=\lrc{\bol{x}\in\mathbb{R}^3:r=0}$. 
Since $\epsilon\sin^2\zeta\leq\epsilon$, 
we have 
\eq{M'=\lrc{\bol{y}\in\lr{-\infty,\Psi_0}\cp[0,2\pi^2):\rho>\epsilon}\subset M.}
On the other hand, observe that
\eq{
\abs{\delta x}=\abs{\cos\zeta\lr{1-\frac{1}{3}\cos^2\zeta}}\leq \frac{2}{3} <1,~~~~\abs{\delta y}=\frac{1}{3}\epsilon\abs{\sin^3\zeta}\leq \frac{1}{3}<1.\label{dxdy}
}
Hence, the condition $\rho>\epsilon$ implies that
\eq{
r^2-2\epsilon\lr{x\delta x+y\delta y}+\epsilon^2\lr{\delta x^2+\delta y^2}>\epsilon^2, 
}
which can be satisfied by demanding that 
\eq{
r^2-2\epsilon r\lr{\frac{2}{3}+\frac{1}{3}}>\epsilon^2.
}
This inequality leads to the condition
\eq{
r>\epsilon\lr{1+\sqrt{2}}.
}
This result shows that the image $N$ corresponds to $\mathbb{R}^3$ minus the `perturbed vertical axis'  $\rho=\epsilon\sin^2\zeta$. In particular, the coordinate change $\bol{x}\lr{\Psi,\zeta,\Theta}$ is well defined in 
\eq{N'=\lrc{\bol{x}\in\mathbb{R}^3:r>\epsilon\lr{1+\sqrt{2}}}\subset N.}
\end{proof}
We now wish to show that, for small enough $\epsilon>0$, level sets of $\Psi$ are diffeomorphic to the torus $\mathbb{T}^2$. To this end, first note that from \eqref{sys16} we have 
\eq{
\Psi=\Psi_0-\lr{\rho-r_0}^2-z^2=\Psi_0-\lrc{\sqrt{\lrs{x+\epsilon\cos\zeta\lr{1-\frac{1}{3}\cos^2\zeta}}^2+\lr{y+\frac{1}{3}\epsilon\sin^3\zeta}^2}-r_0}^2-z^2.\label{Psi}
}
Furthermore, we have the following:



\begin{lemma} 
Let $\Omega'\subset\mathbb{R}^n$,  denote a bounded open set. Suppose that the function $f\in C^{\infty}\lr{\bar{\Omega}'}$ has no critical points, i.e. $\nabla f\neq\bol{0}$ in $\bar{\Omega}'$. 
Take $g\in C^{\infty}\lr{\bar{\Omega}'}$. 
Then, for small enough $\epsilon>0$, $f+\epsilon g$ and $f$ are  diffeomorphic in any closed set $\bar{\Omega}\subset\Omega'$: 
there exists a diffeomorphism $\Phi: 
\bar{\Omega}\rightarrow \bar{\Gamma}\subset\Omega'$, 
such that
\eq{
\Phi^{\ast}\lr{f+\epsilon g}=f,
}
where $\Phi^{\ast}$ denotes the pullback operator. 
\end{lemma}

\begin{proof}
We apply Moser's method: we look for a coordinate change generated by a vector field $X_t\in T\Omega'$ such that
\eq{
X_t\lr{\Phi_t\lr{\bol{x}_0}}=\frac{d}{dt}\Phi_t\lr{\bol{x}_0},~~~~\Phi_0\lr{\bol{x}_0}=\bol{x}_0,\label{dxdt}
}
and
\eq{
\Phi^{\ast}_t\lr{f+t\epsilon' g}=f,
}
where $\Phi_t$ is the flow map, $\bol{x}_0\in\Omega'$, $t\in[0,1]$, and $\epsilon'\in\mathbb{R}_{>0}$. Consider the function $f+t\epsilon'  g$. For small enough $\epsilon'>0$ the quantity $\abs{\nabla\lr{f+t\epsilon' g}}^2 
$ can be made different from zero, since both $\abs{\nabla f}$ and $\abs{\nabla g}$ attain their extrema in $\bar{\Omega}'$ and by hypothesis ${\rm min}_{\bar{\Omega}'}\abs{\nabla f}>0$. Next, observe that
\eq{
\frac{d}{dt}\Phi^{\ast}\lr{f+t\epsilon' g}=X_t\cdot\nabla\lr{f+t\epsilon' g}+\epsilon' g.
}
This quantity can be made zero by choosing
\eq{
X_t=-\epsilon' g\frac{\nabla\lr{f+t\epsilon' g}}{\abs{\nabla\lr{f+t\epsilon' g}}^2}.
}
Since $\abs{\nabla\lr{f+t\epsilon' g}}\neq 0$ and all quantities involved are smooth in ${\Omega}'\times[0,1]
$, the vector field $X_t$ is smooth in ${\Omega}'\times[0,1]
$. Then, from the 
Cauchy-Lipschitz theorem, for each $\bol{x}_0\in\Omega'$ there exists a  $t_{\bol{x}_0}>0$ such that equation \eqref{dxdt} has a solution $\Phi_t\lr{\bol{x}_0}$ in $t\in[0,t_{\bol{x}_0}]$ and the integral curve belongs to $\Omega'$. 
Let 
$t_{\epsilon}>0$ 
denote the minimum of the maximum interval of existence in $\bar{\Omega}$, where  $\bar{\Omega}\subset\Omega'$ is a closed set. Note that the solution  $\Phi_t\lr{\bol{x}_0}$ exists in $\lrs{0,t_{\epsilon}}$ for all $\bol{x}_0\in\bar{\Omega}$. 
Then, the desired diffeomorphism is given by $\Phi=\Phi_{t_\epsilon}$ if $t_{\epsilon}<1$, and by $\Phi=\Phi_1$ if $t_{\epsilon}\geq 1$. Furthermore, the constant $\epsilon$ is given by   $\epsilon=t_{\epsilon}\epsilon'$ in the first case, and by $\epsilon=\epsilon'$ in the second one. Finally, we have $\bar{\Gamma}=\Phi\lr{\bar{\Omega}}$.  
\end{proof}

By using lemma 1, we can prove the following: 
\begin{proposition}
Let $\Omega\subset N'$ denote a bounded region. 
Then, for small enough $\epsilon>0$, level sets of $\Psi=\Psi_0-\lr{\rho-r_0}^2-z^2$ defined in eq. \eqref{Psi} are diffeomorphic to level sets of $\Psi_A=\Psi_0-\lr{r-r_0}^2-z^2$ in $\Omega$.
\end{proposition}

\begin{proof}
To prove this statement it is sufficient to show that level sets of the squared perturbed polar radius $\rho^2$ are diffeomorphic to  cylinders in $\Omega$. 
Indeed, if this were true the perturbed polar radius $\rho$ would define a notion of radial distance, forcing $\Psi$ to be a measure of the distance from the perturbed toroidal axis located at  $\lr{\rho=r_0,z=0}$.  
We have 
\eq{
\rho^2=r^2-2\epsilon\lr{x\delta x+y\delta y}+\epsilon^2\lr{\delta x^2+\delta y^2}.
}
Hence, by setting $g=-2\lr{x\delta x+y\delta y}+\epsilon\lr{\delta x^2+\delta y^2}$, the hypothesis of lemma 1 are verified with $f=r^2$. It follows that $\rho^2$ and $r^2$ are diffeomorphic and the statement is proven. 
\end{proof}

\begin{remark} 
The fact that level sets of the function $\Psi$ given in eq. \eqref{Psi} are diffeomorphic to level sets of $\Psi_A$ can also be deduced from stability theory of smooth mappings \cite{GG}. 
If $f\in C^{\infty}\lr{\bar{\Omega}}$, with $\Omega\subset\mathbb{R}^n$ a smooth bounded domain, is a real valued function with no critical points, it is a stable  Morse function.
Take a real valued function $g\in C^{\infty}\lr{\bar{\Omega}}$. 
Then, for small enough $\epsilon>0$, the function $h=f+\epsilon g$ belongs to a neighborhood of $f$ in the Whitney $C^{\infty}\lr{\Omega}$ topology. Since $f$ is stable, $h$ and $f$ are equivalent, i.e. there exist diffeomorphisms $\Phi_\Omega:\Omega\rightarrow\Omega$ and $\Phi_{\mathbb{R}}:\mathbb{R}\rightarrow\mathbb{R}$ such that $\Phi_{\mathbb{R}}\lr{h\lr{\Phi_{\Omega}\lr{\bol{x}}}}=f\lr{\bol{x}}$ for all $\bol{x}\in\Omega$. 
In the present setting the desired result can be obtained  with $f=r^2$ and $h=\rho^2$. 
\end{remark}


\begin{proposition}
Level sets of the function $\Psi$ given in eq. \eqref{Psi} are not invariant under continuous Euclidean isometries. 
\end{proposition}

\begin{proof}
It is suffcient to show that
\eq{
\mf{L}_{\bol{\xi}_E}\Psi\neq 0,~~~~\forall \bol{\xi}_E\neq\bol{0},
}
where $\bol{\xi}_E=\bol{a}+\bol{b}\cp\bol{x}$, $\bol{a},\bol{b}\in\mathbb{R}^3$, is the generator of continuous Euclidean isometries in three-dimensional Euclidean space. Note that $\rho=\rho\lr{x,y}$ and $\zeta=\zeta\lr{x,y}$ since from \eqref{sys16} we have 
\eq{
\rho^2=\lr{x-\epsilon\delta x\lr{\zeta}}^2+\lr{y-\epsilon\delta y\lr{\zeta}}^2,~~~~\zeta=\arctan\lr{\frac{y-\epsilon\delta y\lr{\zeta}}{x-\epsilon\delta x\lr{\zeta}}}
.}
It follows that
\eq{
\mf{L}_{\bol{\xi}_E}\Psi=&-2\lr{a_x+b_yz-b_zy}\lrs{x-\epsilon\delta x-\epsilon\lr{x-\epsilon\delta x}\frac{\p\delta x}{\p x}-\epsilon\lr{y-\epsilon\delta y}\frac{\p\delta y}{\p x}}\lr{1-\frac{r_0}{\rho}}\\&
-2\lr{a_y+b_zx-b_xz}\lrs{y-\epsilon\delta y-\epsilon\lr{x-\epsilon\delta x}\frac{\p\delta x}{\p y}-\epsilon\lr{y-\epsilon\delta y}\frac{\p\delta y}{\p y}}\lr{1-\frac{r_0}{\rho}}\\&
-2\lr{a_z+b_xy-b_yx}z.\label{Lie1}
}
Now observe that from \eqref{delta} we have
\eq{
\lr{x-\epsilon\delta x}\nabla\delta x+\lr{y-\epsilon\delta y}\nabla\delta y=\lr{x-\epsilon\delta x}\lr{\delta x_{\zeta}+\tan\zeta\delta y_{\zeta}}\nabla\zeta=\bol{0}.
}
Equation \eqref{Lie1} therefore reduces to
\eq{
\mf{L}_{\bol{\xi}_E}\Psi=&-2\lr{1-\frac{r_0}{\rho}}\lrs{\lr{a_x+b_yz-b_zy}\lr{x-\epsilon\delta x}
+\lr{a_y+b_zx-b_xz}\lr{y-\epsilon\delta y}}\\&
-2\lr{a_z+b_xy-b_yx}z.
}
Consider the curve  $\zeta=\pi/2$, $\Theta=0$. We have $x=\delta x=z=0$, $y-\epsilon\delta y=r_0+\sqrt{\Psi_0-\Psi}$. 
For the Lie-derivative \eqref{Lie1} to vanish on this curve we must therefore satisfy the condition
\eq{
-2{a_y}\lr{y-\epsilon\delta y}\lr{1-\frac{r_0}{\rho}}=0,
}
which implies $a_y=0$. 
Now consider the curve $\zeta=0$, $\Theta=0$. We have $y=\delta y=z=0$, $x-\epsilon\delta x=r_0+\sqrt{\Psi_0-\Psi}$. 
For the Lie-derivative \eqref{Lie1} to vanish on this curve we must therefore satisfy the condition
\eq{
-2{a_x}\lr{x-\epsilon\delta x}\lr{1-\frac{r_0}{\rho}}=0,
}
which implies $a_x=0$. 
Next, consider the surface $\Theta=0$. We have $z=0$ and the condition
\eq{
&
2\epsilon b_z\lr{1-\frac{r_0}{\rho}}\lr{ x\delta y-y\delta x}=\frac{4}{3}\epsilon b_z\lr{\rho-r_0}\sin\zeta\cos\zeta=0,
}
which implies $b_z=0$. 
The remaining terms in the Lie-derivative \eqref{Lie1} thus vanish if and only if 
\eq{
&
b_x\lrc{\lr{y-\epsilon\delta y}\lr{1-\frac{r_0}{\rho}}-y}-b_y\lrc{\lr{x-\epsilon\delta x}\lr{1-\frac{r_0}{\rho}}-x}=a_z.
}
For $b_x,b_y\neq 0$, the quantity on the left-hand side is a function. This can be seen, for example, as follows: for $r_0>\epsilon$, at $\Theta=\pi/2$ we have $\rho=r_0>\epsilon$, implying
\eq{
b_y\lr{r_0\cos\zeta+\epsilon\delta x}-b_x\lr{r_0\sin\zeta+\epsilon\delta y}=a_z,
}
which can be satisfied only if $b_x=b_y=a_z=0$. 
We conclude that $\bol{a}=\bol{b}=\bol{0}$.
\end{proof}

\begin{proposition} 
Let $\Omega\subset N'$ denote a toroidal volume enclosed by a level set $\p\Omega$ of the function $\Psi$ as given by equation \eqref{Psi}. Assume $\bar{\Omega}\subset N'$.  
Then, the vector field 
$\bol{u}=\nabla\Phi+\Psi\nabla\Theta+\alpha\nabla\zeta=-\p_{\Theta}$, with $\Theta$ and $\zeta$ given  by equation \eqref{sys16}, $\Phi$ by equation \eqref{Phi}, $\alpha$ by equation \eqref{alpha}, and $\chi_{\Psi}$ by equation \eqref{chipsi},   
is smooth in $\Omega$. Furthermore, it is a solution of system \eqref{sys1}. 
\end{proposition}

\begin{proof}
Since $\bol{u}=-\p_{\Theta}$, the smoothness of $\bol{u}$ in $\Omega$ 
follows from the regularity of the tangent vector $\p_{\Theta}=\p\bol{x}/\p\Theta$. 
The fact that $\bol{u}$ solves system \eqref{sys1} follows from proposition 1.
\end{proof}





\begin{remark} The solution obtained  in proposition 5 gives a contructive proof of theorem 1.  
\end{remark}

\section{Characterization and visualization of solutions}
As a consistency check we first want to show that in the limit $\epsilon\rightarrow 0$ the solution $\bol{u}=-\p_{\Theta}$ constructed in section 3  from the coordinate transformation eq. \eqref{sys16} reduces to the tangent vector in the poloidal direction $\Theta_A=\arctan\lrs{z/\lr{r-r_0}}$, and that in the same limit $\Psi$ tends to the axially symmetry torus, i.e. $\lim_{\epsilon\rightarrow 0}\Psi=\Psi_A=\Psi_0-\lr{r-r_0}^2-z^2$. From  \eqref{xtheta} we find
\eq{
\bol{u}=-\p_{\Theta}=
\frac{z}{\rho}\lr{x-\epsilon\delta x}\nabla x+
\frac{z}{\rho}
\lr{y-\epsilon\delta y}\nabla y
-\lr{\rho-r_0}\nabla z.\label{usec5}
}
On the other hand, $\lim_{\epsilon\rightarrow 0}\lr{\rho,\epsilon\delta x,\epsilon\delta y}=\lr{r,0,0}$, so that
\eq{
\bol{u}_A=\lim_{\epsilon \rightarrow0}\bol{u}=z\nabla r-\lr{r-r_0}\nabla z=-\frac{\nabla\Theta_A}{\abs{\nabla\Theta_A}^2}, 
}
where $\Theta_A=\lim_{\epsilon\rightarrow 0}\Theta$. 
It follows that
\eq{
\lr{\nabla\cp\bol{u}_A}\cp\bol{u}_A=\nabla\Psi_A.
}
Note also that $\bol{u}\neq \bol{u}_A$ for $\epsilon>0$ since $\lr{\nabla\cp\bol{u}}\cp\bol{u}=\nabla\Psi\neq\nabla\Psi_A$. 

It is not difficult to obtain other solutions of system \eqref{sys13} by choosing different $f\lr{\zeta}$ in \eqref{delta} and by 
using the corresponding coordinate change 
\eqref{sys15b} to determine  $\lr{\Psi,\Theta,\zeta}$.
For example, setting $f\lr{\zeta}=\sin^2\lr{3\zeta}$ gives
\eq{
&\delta x=-\frac{1}{2}\cos\zeta-\frac{1}{20}\cos\lr{5\zeta}+\frac{1}{28}\cos\lr{7\zeta},\\&\delta y=-\frac{1}{2}\sin\zeta+\frac{1}{20}\sin\lr{5\zeta}+\frac{1}{28}\sin\lr{7\zeta},\label{f2}
}
while $f\lr{\zeta}=\sin^4\lr{2\zeta}+\cos^2\lr{3\zeta}$ leads to
\eq{
&\delta x=-\frac{7}{8}\cos\zeta-\frac{1}{12}\cos\lr{3\zeta}+\frac{1}{10}\cos\lr{5\zeta}-\frac{3}{112}\cos\lr{7\zeta}-\frac{1}{144}\cos\lr{9\zeta},\\
&\delta y=-\frac{7}{8}\sin\zeta+\frac{1}{12}\sin\lr{3\zeta}-\frac{5}{112}\sin\lr{7\zeta}-\frac{1}{144}\sin\lr{9\zeta}.\label{f3}
}
It is also possible to identify more general classes of perturbations that leave the coordinate change $\bol{x}\lr{\Psi,\Theta,\zeta}$ consistent with \eqref{sys13}. For example, take  $\epsilon_1,\epsilon_2,\epsilon_3$ to be positive real constants and consider the coordinate change
\sys{
&x=\lr{r_0+\sqrt{\Psi_0-\Psi}\cos\Theta+\epsilon_2\sin\Theta-\epsilon_3 g\lr{\Theta}}\cos\zeta+\epsilon_1\delta x\lr{\zeta},\\
&y=\lr{r_0+\sqrt{\Psi_0-\Psi}\cos\Theta+\epsilon_2\sin\Theta-\epsilon_3g\lr{\Theta}}\sin\zeta+\epsilon_1\delta y\lr{\zeta},\\
&z=\sqrt{\Psi_0-\Psi}\sin\Theta+\epsilon_2\cos\Theta+\epsilon_3\delta z\lr{\Theta},
}{cc1}
where the functions $\delta z\lr{\Theta}$ and $g\lr{\Theta}$ are related by the following differential equation in $\Theta$, 
\eq{
\sin\Theta\lr{\frac{dg}{d\Theta}-\frac{d^2\delta z}{d\Theta^2}}+\cos\Theta\lr{\frac{d\delta z}{d\Theta}+\frac{d^2g}{d\Theta^2}}=0,\label{cc12}
}
and the perturbations $\delta x$ and $\delta y$ are given by \eqref{delta}. 
One can verify that \eqref{cc1} is consistent with system \eqref{sys13}. 
An example of solution of eq. \eqref{cc12} is
\eq{
\delta z=-\sin\lr{4\Theta}-10\sin\lr{2\Theta},~~~~g=\cos\lr{4\Theta}-10\cos\lr{2\Theta}.
}

Figure 1 shows plots of the level sets $\Psi=\Psi_c\in\mathbb{R}$ 
of the function $\Psi\lr{\bol{x}}$ 
obtained through the parametric surfaces $\bol{x}\lr{\Psi_c,\Theta,\zeta}$ that result from fixing $\Psi$ and  varying the parameters $\Theta,\zeta\in[0,2\pi)$ in the coordinate change \eqref{cc1}. Observe that the perturbations scaled by the constants $\epsilon_1,\epsilon_2,\epsilon_3$ break the symmetry of the axially symmetric (unperturbed) toroidal surface shown in figure 1(a).     
\begin{figure}[h]
\hspace*{-0cm}\centering
\includegraphics[scale=0.5]{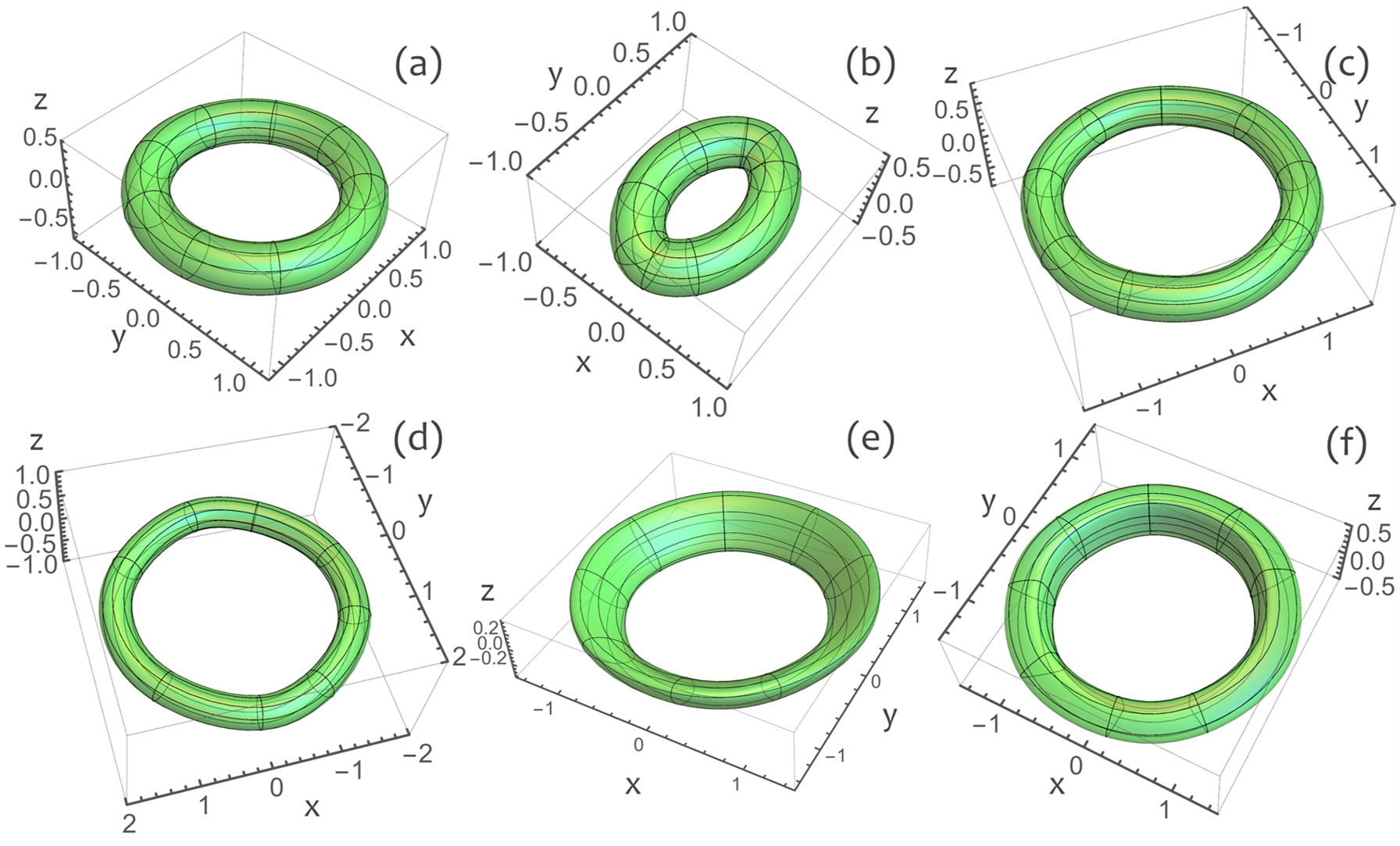}
\caption{\footnotesize 
Parametric plots of the level set $\Psi=0.95$ corresponding to \eqref{cc1} for 
(a) $\delta x=0$, $\delta y=0$, $\epsilon_1=0$, $\epsilon_2=0$, $\epsilon_3=0$, (b) 
$\delta x=-\cos\zeta\lr{1-\frac{1}{3}\cos^2\zeta}$, $\delta y=-\frac{1}{3}\sin^3\zeta$, $\epsilon_1=0.8$, $\epsilon_2=0$, $\epsilon_3=0$, 
(c) 
$\delta x=-\frac{1}{2}\cos\zeta-\frac{1}{20}\cos\lr{5\zeta}+\frac{1}{28}\cos\lr{7\zeta}$, $\delta y=-\frac{1}{2}\sin\zeta+\frac{1}{20}\sin\lr{5\zeta}+\frac{1}{28}\sin\lr{7\zeta}$, $\epsilon_1=0.7$, $\epsilon_2=0$, $\epsilon_3=0$, 
(d) 
$\delta x=-\frac{7}{8}\cos\zeta-\frac{1}{12}\cos\lr{3\zeta}+\frac{1}{10}\cos\lr{5\zeta}-\frac{3}{112}\cos\lr{7\zeta}-\frac{1}{144}\cos\lr{9\zeta}$, $\delta y=-\frac{7}{8}\sin\zeta+\frac{1}{12}\sin\lr{3\zeta}-\frac{5}{112}\sin\lr{7\zeta}-\frac{1}{144}\sin\lr{9\zeta}$, $\epsilon_1=0.7$, $\epsilon_2=0$, $\epsilon_3=0$,
(e) 
$\delta x=-\frac{1}{2}\cos\zeta-\frac{1}{20}\cos\lr{5\zeta}+\frac{1}{28}\cos\lr{7\zeta}$, $\delta y=-\frac{1}{2}\sin \zeta+\frac{1}{20}\sin\lr{5\zeta}+\frac{1}{28}\sin\lr{7\zeta}$, $\epsilon_1=0.5$, $\epsilon_2=0.1$, $\epsilon_3=0$, 
(f) 
$\delta x=-\frac{1}{2}\cos\zeta-\frac{1}{20}\cos\lr{5\zeta}+\frac{1}{28}\cos\lr{7\zeta}$, $\delta y=-\frac{1}{2}\sin \zeta+\frac{1}{20}\sin\lr{5\zeta}+\frac{1}{28}\sin\lr{7\zeta}$, $\delta z=-\sin\lr{4\Theta}-10\sin\lr{2\Theta}$, $g=\cos\lr{4\Theta}-10\cos\lr{2\Theta}$, $\epsilon_1=0.5$, $\epsilon_2=0.05$, $\epsilon_3=0.005$. In these plots $r_0=1$ and $\Psi_0=1$.}
\label{fig1}
\end{figure}

We now return to the solution $\bol{u}=-\p_{\Theta}$ constructed in section 3 from the coordinate transformation eq. \eqref{sys16}. To visualize this vector field, we consider a small perturbation $\epsilon <<1$. 
Recalling \eqref{dxdy}, we have $\delta x=O\lr{1}$ as well as $\delta y=O\lr{1}$. Furthermore, 
\eq{
\zeta=&\arctan\lr{\frac{y-\epsilon\delta y}{x-\epsilon \delta x}}=\arctan\lr{\frac{y}{x}}+\frac{\epsilon}{r^2}\lr{y\delta x-x\delta y}+O\lr{\epsilon^2}\\=&\varphi+\frac{\epsilon}{r^2}\lrs{\frac{1}{3}x\sin^3\varphi-y\cos\varphi\lr{1-\frac{1}{3}\cos^2\varphi}}+O\lr{\epsilon^2}\\=&\varphi-\frac{2\epsilon}{3r}\sin\varphi\cos\varphi+O\lr{\epsilon^2},
}
and
\eq{
\rho=r\lrs{1-\frac{\epsilon}{r^2}\lr{x\delta x+y\delta y}}+O\lr{\epsilon^2}=
r+\frac{\epsilon}{3}\lr{1+\cos^2\varphi}+O\lr{\epsilon^2}.
}
Now define
\eq{
&\epsilon\delta r=\rho-r=\epsilon\delta r_0+O\lr{\epsilon^2},~~~~\delta r_0=
\frac{\epsilon}{3}\lr{1+\cos^2\varphi},\\&\delta x_0=-\cos\varphi\lr{1-\frac{1}{3}\cos^2\varphi},~~~~\delta y_0=-\frac{1}{3}\sin^3\varphi. 
}
Using these expressions, we see that the vector field \eqref{usec5} can be expanded at first order in $\epsilon$ as
\eq{
\bol{u}=&\frac{z}{r}\lr{1+\epsilon\frac{\delta r}{r}}\lr{x-\epsilon\delta x}\nabla x
+\frac{z}{r}\lr{1+\epsilon\frac{\delta r}{r}}\lr{y-\epsilon\delta y}\nabla y-\lr{r-r_0-\epsilon\delta r}\nabla z+O\lr{\epsilon^2}\\
    =&\frac{z}{r}\lrs{x+\epsilon\lr{x\frac{\delta r_0}{r}-\delta x_0}}\nabla x+\frac{z}{r}\lrs{y+\epsilon\lr{y\frac{\delta r_0}{r}-\delta y_0}}\nabla y-\lr{r-r_0-\epsilon\delta r_0}\nabla z+O\lr{\epsilon^2},\label{uTay}
}
and that the function $\Psi$ of eq. \eqref{Psi} can be approximated as
\eq{
\Psi=\Psi_0-\lr{r-r_0}^2-z^2-2\epsilon\lr{r-r_0}\delta r_0+O\lr{\epsilon^2}.\label{PsiTay}
}

Figure 2 shows plots of the first-order expansion $\bol{u}_1$ of the vector field $\bol{u}=\bol{u}_1+O\lr{\epsilon^2}$ given in eq. \eqref{uTay} on the level sets of the first order expansion $\Psi_1$ of the function $\Psi=\Psi_1+O\lr{\epsilon^2}$ given in eq. \eqref{PsiTay} for different values of the ordering parameter $\epsilon$. 
Observe that the tangential condition $\bol{u}\cdot\nabla\Psi=0$ 
for the solution $\bol{u}$ 
is progressively violated for larger values of $\epsilon$, since the accuracy of the Taylor's expansions of both $\bol{u}$ and $\Psi$ deteriorates (see figure 1 for the non-approximated level sets of $\Psi$).  
From these figures, one sees that the solution $\bol{u}=-\p_{\Theta}$ exhibits a `poloidal' behavior, i.e. the flow tends to loop around the toroidal axis 
within level sets of the toroidal angle $\zeta$.

\begin{figure}[h]
\hspace*{-0cm}\centering
\includegraphics[scale=0.6]{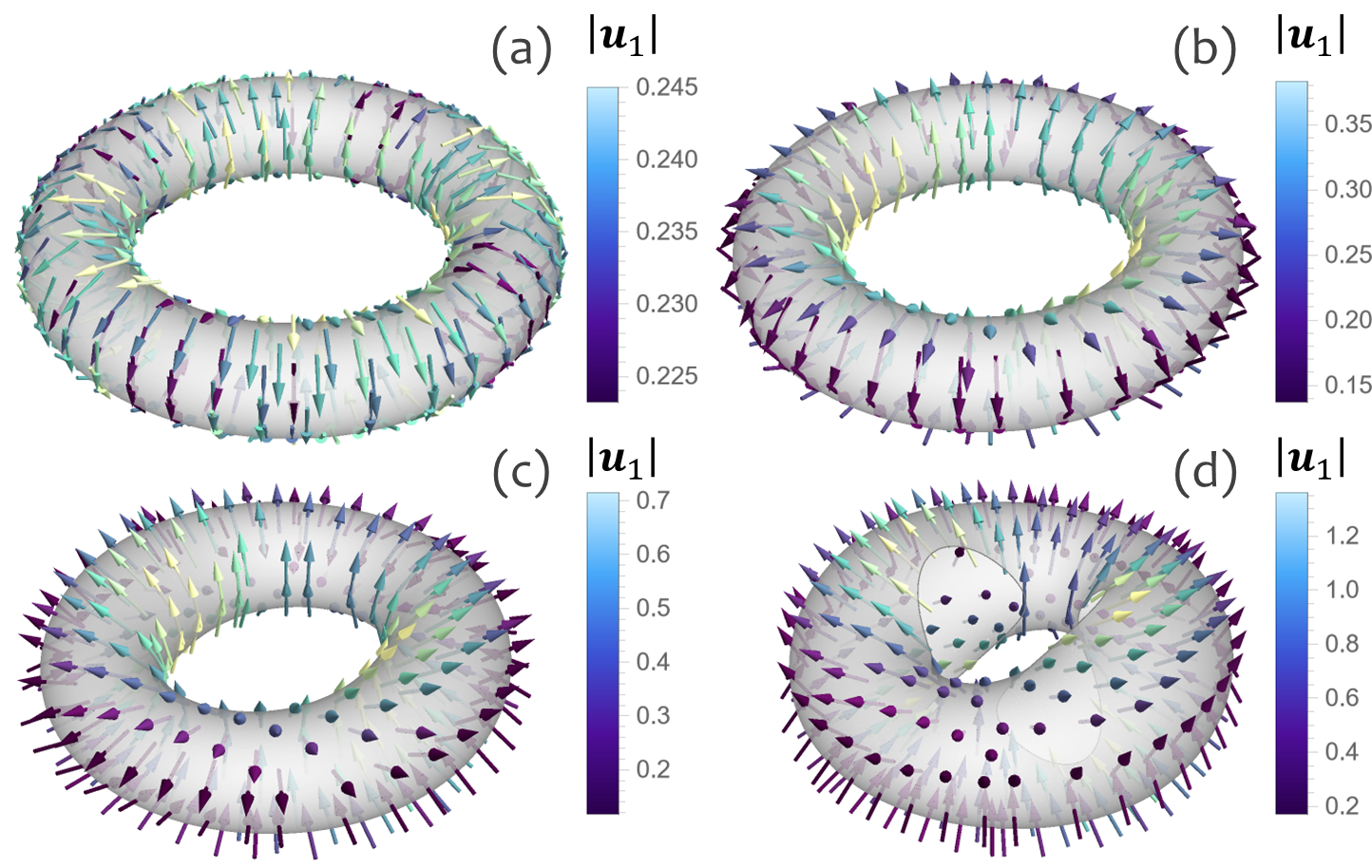}
\caption{\footnotesize 
Plots of the first-order expansion $\bol{u}_1$ of the vector field $\bol{u}=\bol{u}_1+O\lr{\epsilon^2}$ given in eq. \eqref{uTay} on the level sets of the first order expansion $\Psi_1$ of the function $\Psi=\Psi_1+O\lr{\epsilon^2}$ given in eq. \eqref{PsiTay} for $\delta x=-\cos\zeta\lr{1-\frac{1}{3}\cos^2\zeta}$ and $\delta y=-\frac{1}{3}\sin^3\zeta$. Recall that $\bol{u}$ solves \eqref{sys1} with $P=\Psi$.  
(a) $\epsilon=0$. (b) $\epsilon=0.1$. (c). $\epsilon=0.3$ (d) $\epsilon=0.6$.}
\label{fig2}
\end{figure}



\section{Remarks on Incompressibile Solutions}

The aim of this last section is to discuss some consequences of the present theory for the case in which solenoidal solutions ($\nabla\cdot\bol{u}=0$) of system \eqref{sys1} are sought in the `poloidal' form $\bol{u}=-\p_{\Theta}$. 
First, recall that 
$\p_{\Theta}\cdot\nabla\Psi=\p_{\Theta}\cdot\nabla\zeta=0$ and $\p_{\Theta}\cdot\nabla\Theta=1$. Hence, 
\eq{
\nabla\cdot\bol{u}=-\nabla\cdot\lr{\p_{\Theta}}=-J\frac{\p}{\p\Theta}\lr{\frac{1}{J}}=\frac{1}{J}\frac{\p J}{\p\Theta}=0.\label{divu0}
}
Noting that $J^{-1}=\p_{\Psi}\cdot\p_{\Theta}\cp\p_{\zeta}$, equation \eqref{divu0} is thus equivalent to
\eq{
\frac{\p^2\bol{x}}{\p\Theta\p\Psi}\cdot\frac{\p\bol{x}}{\p\Theta}\cp\frac{\p\bol{x}}{\p\zeta}+
\frac{\p\bol{x}}{\p\Psi}\cdot\frac{\p^2\bol{x}}{\p\Theta^2}\cp\frac{\p\bol{x}}{\p\zeta}+
\frac{\p\bol{x}}{\p\Psi}\cdot\frac{\p\bol{x}}{\p\Theta}\cp\frac{\p^2\bol{x}}{\p\Theta\p\zeta}=0.\label{divu02}
}
Equation \eqref{divu02} can be written as
\eq{
\frac{\p^2\bol{x}}{\p\Theta\p\Psi}&\cdot\nabla\Psi+\frac{\p^2\bol{x}}{\p\Theta^2}\cdot\nabla\Theta+\frac{\p^2\bol{x}}{\p\Theta\p\zeta}\cdot\nabla\zeta=
g^{\Psi\Psi}\frac{\p^2\bol{x}}{\p\Theta\p\Psi}\cdot\frac{\p\bol{x}}{\p\Psi}+g^{\Psi\Theta}\lr{\frac{\p^2\bol{x}}{\p\Theta\p\Psi}\cdot\frac{\p\bol{x}}{\p\Theta}+\frac{\p^2\bol{x}}{\p\Theta^2}\cdot\frac{\p\bol{x}}{\p\Psi}}\\&
+g^{\Psi\zeta}\lr{\frac{\p^2\bol{x}}{\p\Theta\p\Psi}\cdot\frac{\p\bol{x}}{\p\zeta}+\frac{\p^2\bol{x}}{\p\Theta\p\zeta}\cdot\frac{\p\bol{x}}{\p\Psi}}
+g^{\Theta\Theta}{\frac{\p^2\bol{x}}{\p\Theta^2}\cdot\frac{\p\bol{x}}{\p\Theta}}
\\&+g^{\Theta\zeta}\lr{\frac{\p^2\bol{x}}{\p\Theta^2}\cdot\frac{\p\bol{x}}{\p\zeta}+\frac{\p^2\bol{x}}{\p\Theta\p\zeta}\cdot\frac{\p\bol{x}}{\p\Theta}}+g^{\zeta\zeta}\frac{\p^2\bol{x}}{\p\Theta\p\zeta}\cdot\frac{\p\bol{x}}{\p\zeta}\\=&\frac{1}{2}g^{\Psi\Psi}\frac{\p g_{\Psi\Psi}}{\p\Theta}+g^{\Psi\Theta}\frac{\p g_{\Psi\Theta}}{\p\Theta}+g^{\Psi\zeta}\frac{\p g_{\Psi\zeta}}{\p\Theta}+\frac{1}{2}g^{\Theta\Theta}\frac{\p g_{\Theta\Theta}}{\p\Theta}+g^{\Theta\zeta}\frac{\p g_{\Theta\zeta}}{\p\Theta}+\frac{1}{2}g^{\zeta\zeta}\frac{\p g_{\zeta\zeta}}{\p\Theta}
\\=&g^{\Psi\Psi}\Gamma_{\Psi\Psi\Theta}+g^{\Psi\Theta}\lr{\Gamma_{\Psi\Theta\Theta}+\Gamma_{\Theta\Psi\Theta}}+g^{\Psi\zeta}\lr{\Gamma_{\Psi\zeta\Theta}+\Gamma_{\zeta\Theta\Psi}}+g^{\Theta\Theta}\Gamma_{\Theta\Theta\Theta}+g^{\Theta\zeta}\lr{\Gamma_{\zeta\Theta\Theta}+\Gamma_{\Theta\zeta\Theta}}+g^{\zeta\zeta}\Gamma_{\zeta\zeta\Theta}\\=&\Gamma^{\Psi}_{\Psi\Theta}+\Gamma^{\Theta}_{\Theta\Theta}+\Gamma^{\zeta}_{\zeta\Theta}=0,\label{Chris2}
}
where $\Gamma^{i}_{jk}=\frac{1}{2}g^{im}\lr{\frac{\p g_{mj}}{\p x^k}+\frac{\p g_{mk}}{\p x^j}-\frac{\p g_{jk}}{\p x^m}}$ are the Christoffel symbols of the second kind. 

These results can be used to state the following: 

\begin{proposition}
Let $\lr{\Psi,\Theta,\zeta}$ be a smooth coordinate system in a smooth bounded domain $\Omega\subset\mathbb{R}^3$ satisfying
\sys{
&\frac{\p\bol{x}}{\p\Psi}\cdot\frac{\p^2\bol{x}}{\p\Theta^2}=1+\frac{\p\bol{x}}{\p\Theta}\cdot\frac{\p^2\bol{x}}{\p\Psi\p\Theta},\\
&\frac{\p\bol{x}}{\p\Theta}\cdot\frac{\p^2\bol{x}}{\p\Theta\p\zeta}=\frac{\p\bol{x}}{\p\zeta}\cdot\frac{\p^2\bol{x}}{\p\Theta^2},\\
&\frac{\p^2\bol{x}}{\p\Theta\p\Psi}\cdot\frac{\p\bol{x}}{\p\Theta}\cp\frac{\p\bol{x}}{\p\zeta}+
\frac{\p\bol{x}}{\p\Psi}\cdot\frac{\p^2\bol{x}}{\p\Theta^2}\cp\frac{\p\bol{x}}{\p\zeta}+
\frac{\p\bol{x}}{\p\Psi}\cdot\frac{\p\bol{x}}{\p\Theta}\cp\frac{\p^2\bol{x}}{\p\Theta\p\zeta}=0.
}{sysdiv02}
Suppose that the boundary $\p\Omega$ corresponds to a contour of the function $\Psi$. 
Then, the vector field $\bol{u}=-\p_{\Theta}$ 
solves \eqref{sys1} with $P=\Psi$ and $\nabla\cdot\bol{u}=0$. 
\end{proposition}

\begin{proof}
The proof of this statement follows from proposition 1 and the equivalence between \eqref{divu0} and \eqref{divu02}. 
\end{proof}
Alternatively, we have
\begin{proposition}
Let $\lr{\Psi,\Theta,\zeta}$ be a smooth coordinate system in a smooth bounded domain $\Omega\subset\mathbb{R}^3$ satisfying
\sys{
&\Gamma_{\Psi\Theta\Theta}=1+\Gamma_{\Theta\Psi\Theta},\\
&\Gamma_{\Theta\zeta\Theta}=\Gamma_{\zeta\Theta\Theta},\\
&\Gamma_{\Psi\Theta}^{\Psi}+\Gamma_{\Theta\Theta}^{\Theta}+\Gamma_{\zeta\Theta}^{\Theta}=0.
}{sysdiv0}
Suppose that the boundary $\p\Omega$ corresponds to a contour of the function $\Psi$. 
Then, the vector field $\bol{u}=-\p_{\Theta}$ 
solves \eqref{sys1} with $P=\Psi$ and $\nabla\cdot\bol{u}=0$. 
\end{proposition}

\begin{proof}
The proof of this statement follows from proposition 1, the equivalence between \eqref{sys13} and eq. \eqref{Chris1}, and the equivalence between \eqref{divu0} and equation \eqref{Chris2}.
\end{proof}

In essence, propositions 6 and 7 reduce the problem of determining a steady incompressible Euler flow (or a magnetohydrodynamic  equilibrium) in a bounded domain to the existence of a coordinate system whose Christoffel symbols satisfy  \eqref{sysdiv0}, and with the additional property that the boundary is given as a level set of one of the coordinates.  
The advantage of this formulation is that, as shown in section 4, solutions of \eqref{sysdiv02} or \eqref{sysdiv0} can sought in the form $\bol{x}\lr{\Psi,\Theta,\zeta}$, which better accommodates the geometry of the problem compared to formulations such as \eqref{sys1} in terms of the vector field  $\bol{u}\lr{\bol{x}}$.
Furthermore, the geometric formulation of the steady Euler equations in terms of the metric tensor enables us to generalize the notion of steady
Euler flow. 
Indeed, recalling equation  \eqref{sys12}, we have 
\begin{mydef}
Let $\lr{\mc{M},g}$ denote a $3$-dimensional Riemannian manifold $\mc{M}$ with metric tensor $g$ 
spanned by coordinates $\lr{\Psi,\zeta,\Theta}$. 
Suppose that 
\sys{
&\frac{\p g_{\Psi\Theta}}{\p\Theta}=1+\frac{\p g_{\Theta\Theta}}{\p\Psi},\\
&\frac{\p g_{\Theta\Theta}}{\p\zeta}=\frac{\p g_{\Theta\zeta}}{\p\Theta}.
}{genE}
Then, $\bol{u}=-\p_{\Theta}$ defines a generalized steady Euler flow on $\mc{M}$. 
The flow is incompressible when
\eq{
\frac{\p\abs{\det\lr{g}}}{\p\Theta}=0.\label{gTh}
}
Furthermore, this flow reduces to a steady solution of the Euler equations 
when there exists a coordinate change to Cartesian coordinates $\bol{x}\lr{\Psi,\Theta,\zeta}$ realizing the metric tensor $g$. 
\end{mydef}

Since the metric tensor $g$ has $6$ independent components, solutions  of \eqref{genE} and \eqref{gTh} can be obtained rather easily. 
The existence of a corresponding realization in $3$-dimensional  Euclidean space is therefore contingent upon the existence of the coordinate change $\bol{x}\lr{\Psi,\Theta,\zeta}$, a problem that can be addressed separately.




\section{Concluding remarks}


In this work, we have investigated the existence of smooth solutions to system \eqref{sys1}. The main result, encapsulated in Theorem 1, demonstrates that there exist smooth solutions of system \eqref{sys1} in toroidal volumes without continuous Euclidean isometries, where the pressure gradient \(\nabla P \neq \bol{0}\). The proof of Theorem 1 employs a constructive approach: first, system \eqref{sys1} is reformulated as a problem of determining a coordinate transformation that satisfies the properties \eqref{sys13}. This transformation is then explicitly derived as a perturbation of an axially symmetric configuration, such as in equation \eqref{sys16}. The solutions obtained are compressible poloidal flows of the form \(\bol{u} = -\p_{\Theta}\), which lack continuous Euclidean isometries and support a non-vanishing pressure gradient.

As discussed in the introduction, this result has significant implications for addressing Grad's conjecture on the existence of regular magnetohydrostatic configurations \eqref{MHS} in toroidal volumes, supported by a non-vanishing pressure gradient, without continuous Euclidean isometries. Specifically, the findings of Theorem 1, alongside the analysis in \cite{SatoYamada23}, indicate that both compressible flows subject to an external potential or supported by a density source and incompressible flows supported by a pressure anisotropy (mathematically described by an integration factor \(\lambda(\bol{x})\) in front of the pressure gradient) can exist without continuous Euclidean isometries. These observations suggest that the primary challenges in solving Grad's conjecture are (i) the incompressibility constraint on the flow (or magnetic field), and (ii) the force balance between rotational and conservative forces. Therefore, the presence of continuous Euclidean isometries may not be the critical factor determining the conditions for the existence of solutions to the magnetohydrostatic equations. In other words, the \(\mathbb{R}^3\) differential operators (curl, gradient, and divergence) in \eqref{MHS} do not appear to inherently impose Euclidean symmetries on the solutions. Instead, these symmetries simplify the system by eliminating its hyperbolic components, thereby reducing it to the Grad-Shafranov equation. Conversely, the exactness and closedness conditions expressed by force balance and incompressibility seem to be the key obstacles to the existence of regular solutions in asymmetric toroidal configurations.

We also note that the formulation of the steady Euler/magnetohydrostatic equations, as presented in Propositions 6 and 7, may offer an alternative approach for examining the existence of solutions to system \eqref{MHS}. Notably, system \eqref{sysdiv02} is a second-order system of three partial differential equations for three unknowns \(\bol{x}(\Psi, \Theta, \zeta)\). Moreover, the method of Clebsch reduction has proven to be an effective tool for analyzing systems \eqref{sys1}, \eqref{sys2}, and \eqref{MHS} (see also \cite{SatoYamadaBoh}).

Finally, it is important to emphasize that the solutions to system \eqref{sys1} obtained in this study lack rotational transform (the number of poloidal transits per toroidal transit). Additionally, the explicit examples provided exhibit reflection symmetry across a plane. Whether solutions with rotational transform or without reflection symmetry exist remains an open question that warrants further investigation.

\section*{Acknowledgment}
N.S. would like to thank D. Pfefferl\'e, D. Perrella, N. Duignan, and F. Pasqualotto for useful discussion.   

\appendix

\section{Clebsch reduction of system (1)}
The aim of this appendix section is to reduce system \eqref{sys1} into a set of equations for Clebsch potentials. 
To this end, it is useful to write system \eqref{sys1} in the equivalent form
\sys{
&\bol{q}\cp\bol{u}=\nabla P~~~~{\rm in}~~\Omega,\\
&\bol{q}=\nabla\cp\bol{u}~~~~{\rm in}~~\Omega,\\
&\bol{u}\cdot\bol{n}=0,~~~~
P=P_b~~~~{\rm on}~~\p\Omega,
}{sys3}
where we introduced the auxiliary vector field $\bol{q}\lr{\bol{x}}$. It is standard that every exact $2$-form $q=du\in L^2\lr{\Omega}$, with $u=u_idx^i$, admits the following Clebsch representation (see theorems 2 and 3 of \cite{YosClebsch}),
\eq{
q=d\Psi\w d\Theta+d\alpha\w d\zeta.
}
Here, $\lr{\Psi,\theta,\zeta}\in C^2\lr{S_R}$ are a set of curvilinear coordinates such that the bounding surface corresponds to a level set of $\Psi$, i.e.  $\p\Omega=\lrc{\bol{x}\in S_R:\Psi\lr{\bol{x}}=\Psi_b\in\mathbb{R}}$, $S_R$ denotes a sphere of radius $R$ such that $\bar{\Omega}\subset S_R$, and $\Theta,\alpha\in H^1\lr{\Omega}$ are some functions. 
The quantities $\Psi$, $\Theta$, $\alpha$, and $\zeta$ are called Clebsch potentials. 
In vector notation, we have
\eq{
\bol{q}=\nabla\cp\bol{u}=\nabla\Psi\cp\nabla\Theta+\nabla\alpha\cp\nabla\zeta.
}
Assume that $J=\nabla\Psi\cp\nabla\Theta\cdot\nabla\zeta\neq 0$ in $\bar{\Omega}$ and perform the change of coordinates $\lr{\Psi,\theta,\zeta}\rightarrow\lr{\Psi,\Theta,\zeta}$.  
Then, setting $u^{\Psi}=\bol{u}\cdot\nabla\Psi$, $u^{\Theta}=\bol{u}\cdot\nabla\Theta$, and $u^{\zeta}=\bol{u}\cdot\nabla\zeta$, system \eqref{sys3} takes the form
\sys{
&-\lr{\alpha_{\Psi}u^{\zeta}+u^{\Theta}+P_{\Psi}}\nabla\Psi+\lr{u^{\Psi}-\alpha_{\Theta}u^{\zeta}-P_{\Theta}}\nabla\Theta+\lr{\alpha_{\Psi}u^{\Psi}+u^{\Theta}\alpha_{\Theta}-P_{\zeta}}\nabla\zeta=\bol{0}~~~~{\rm in}~~\Omega,\label{sys4a}\\
&\bol{u}=\nabla\Phi+\Psi\nabla\Theta+\alpha\nabla\zeta~~~~{\rm in}~~\Omega,\label{u}\\
&\bol{u}\cdot\nabla\Psi=0,~~~~
P=P_b
,~~~~\Psi=\Psi_b
~~~~{\rm on}~~{\p\Omega},
}{sys4}
where $\nabla\Phi\in {\rm ker}\lr{{\rm rot}}$ and $\Psi_b\in\mathbb{R}$. 
Introducing the metric coefficients
\eq{
g_{\Psi\Psi}=\p_{\Psi}\cdot\p_{\Psi},~~~~g_{\Psi\Theta}=\p_{\Psi}\cdot\p_{\Theta},~~~~g_{\Psi\zeta}=\p_{\Psi}\cdot\p_{\zeta},~~~~g_{\Theta\Theta}=\p_{\Theta}\cdot\p_{\Theta},~~~~g_{\Theta\zeta}=\p_{\Theta}\cdot\p_{\zeta},~~~~g_{\zeta\zeta}=\p_{\zeta}\cdot\p_{\zeta},\label{mt}
}
where $\lr{\p_{\Psi},\p_{\Theta},\p_{\zeta}}$ are the tangent vectors of the coordinate system, 
equation \eqref{sys4} can be equivalently written as 
\sys{
&u^{\Theta}=-\alpha_{\Psi}u^{\zeta}-P_{\Psi}~~~~{\rm in}~~{\Omega},\\
&u^{\Psi}=\alpha_{\Theta}u^{\zeta}+P_{\Theta}~~~~{\rm in}~~{\Omega},\\
&\alpha_{\Psi}P_{\Theta}-\alpha_{\Theta}P_{\Psi}=P_{\zeta}~~~~{\rm in}~~{\Omega},\\
&u^{\zeta}\lr{g_{\Psi\zeta}-\alpha_{\Psi}g_{\Psi\Theta}+\alpha_{\Theta}g_{\Psi\Psi}}=P_{\Psi}g_{\Psi\Theta}-P_{\Theta}g_{\Psi\Psi}+\Phi_{\Psi}~~~~{\rm in}~~{\Omega},\\
&u^{\zeta}\lr{\alpha_{\Theta}g_{\Psi\Theta}-\alpha_{\Psi}g_{\Theta\Theta}+g_{\Theta\zeta}}=\Phi_{\Theta}+\Psi+P_{\Psi}g_{\Theta\Theta}-P_{\Theta}g_{\Psi\Theta}~~~~{\rm in}~~{\Omega},\\
&u^{\zeta}\lr{g_{\zeta\zeta}+\alpha_{\Theta}g_{\Psi\zeta}-\alpha_{\Psi}g_{\Theta\zeta}}=\alpha+\Phi_{\zeta}+P_{\Psi}g_{\Theta\zeta}-P_{\Theta}g_{\Psi\zeta}~~~~{\rm in}~~{\Omega},\\
&u^{\Psi}=0,~~~~P=P_b
,~~~~\Psi=\Psi_b
~~~~{\rm on}~~\p\Omega.
}{sys5}
Suppose that $g_{\zeta\zeta}+\alpha_{\Theta}g_{\Psi\zeta}-\alpha_{\Psi}g_{\Theta\zeta}\neq 0$ in $\bar{\Omega}$. Then, in $\Omega$ the components $u^{\Psi}$, $u^{\Theta}$, and $u^{\zeta}$ are completely deterimined by
the coordinate system $\lr{\Psi,\Theta,\zeta}$ and the functions $\alpha$, $\Phi$, and $P$:
\sys{
&u^{\Theta}=-\alpha_{\Psi}u^{\zeta}-P_{\Psi},\\
&u^{\Psi}=\alpha_{\Theta}u^{\zeta}+P_{\Theta},\\
&u^{\zeta}=\frac{\alpha+\Phi_{\zeta}+P_{\Psi}g_{\Theta\zeta}-P_{\Theta}g_{\Psi\zeta}}{g_{\zeta\zeta}+\alpha_{\Theta}g_{\Psi\zeta}-\alpha_{\Psi}g_{\Theta\zeta}}.
}{sys6}
System \eqref{sys5} can thus be reduced to the following equations 
for the coordinates $\lr{\Psi,\Theta,\zeta}$ and the unknowns $\alpha$, $\Phi$, and $P$,
\sys{
&\alpha_{\Psi}P_{\Theta}-\alpha_{\Theta}P_{\Psi}=P_{\zeta}~~~~{\rm in}~~{\Omega},\\
&\frac{\alpha+\Phi_{\zeta}+P_{\Psi}g_{\Theta\zeta}-P_{\Theta}g_{\Psi\zeta}}{g_{\zeta\zeta}+\alpha_{\Theta}g_{\Psi\zeta}-\alpha_{\Psi}g_{\Theta\zeta}}\lr{g_{\Psi\zeta}-\alpha_{\Psi}g_{\Psi\Theta}+\alpha_{\Theta}g_{\Psi\Psi}}={P_{\Psi}g_{\Psi\Theta}-P_{\Theta}g_{\Psi\Psi}+\Phi_{\Psi}}~~~~{\rm in}~~{\Omega},\\
&\frac{\alpha+\Phi_{\zeta}+P_{\Psi}g_{\Theta\zeta}-P_{\Theta}g_{\Psi\zeta}}{g_{\zeta\zeta}+\alpha_{\Theta}g_{\Psi\zeta}-\alpha_{\Psi}g_{\Theta\zeta}}\lr{\alpha_{\Theta}g_{\Psi\Theta}-\alpha_{\Psi}g_{\Theta\Theta}+g_{\Theta\zeta}}=\Phi_{\Theta}+\Psi+P_{\Psi}g_{\Theta\Theta}-P_{\Theta}g_{\Psi\Theta}~~~~{\rm in}~~{\Omega},\\
&\alpha_{\Theta}\frac{\alpha+\Phi_{\zeta}+P_{\Psi}g_{\Theta\zeta}-P_{\Theta}g_{\Psi\zeta}}{g_{\zeta\zeta}+\alpha_{\Theta}g_{\Psi\zeta}-\alpha_{\Psi}g_{\Theta\zeta}}+P_{\Theta}=0,~~~~P=P_b,
~~~~\Psi=\Psi_b
~~~~{\rm on}~~\p\Omega.
}{sys7}




\section*{Statements and declarations}

\subsection*{Data availability}
Data sharing not applicable to this article as no datasets were generated or analysed during the current study.

\subsection*{Funding}
The research of NS was partially supported by JSPS KAKENHI Grant No. 21K13851, 22H04936, and 24K00615. This work was partly supported by MEXT Promotion of Distinctive Joint Research Center Program JPMXP0723833165. 

\subsection*{Competing interests} 
The author has no competing interests to declare that are relevant to the content of this article.


\begin{thebibliography}{99}

\bibitem{Hel} P. Helander,
\ti{Theory of plasma confinement in nonaxisymmetric magnetic fields}, Rep. Prog. Phys. \tb{77}, 087001 (2014). 

\bibitem{Kruskal} M. D. Kruskal and R. M. Kulsrud,
\ti{Equilibrium of a magnetically confined plasma in a toroid}, 
The Physics of Fluids \tb{1}, 4 (1958). 

\bibitem{Moffatt85} H. K. Moffatt, 
\ti{Magnetostatic equilibria and analogous Euler flows of arbitrary complex topology. Part 1. Fundamentals}, 
J. Fluid Mech. \tb{159}, pp. 359-378 (1985).



\bibitem{LoSurdo} C. Lo Surdo, 
\ti{Global magnetofluidostatic fields (an unsolved PDE problem)}, Int. J. Math. Math. Sci. \tb{9}, pp. 123-130 (1986).

\bibitem{Grad60} H. Grad, 
\ti{Reducible problems in magneto-fluid dynamics steady flows}, 
Rev. Mod. Phys. \tb{32}, 4 (1960). 

\bibitem{Yos90} Z. Yoshida and H. Yamada
\ti{Structurally-unstable
electrostatic potentials in plasmas},
Prog. Theor. Phys. \tb{84} 2 (1990).

\bibitem{Grad58} H. Grad and H. Rubin, 
\ti{Hydromagnetic equilibria and force-free fields} in Theoretical and experimental aspects of controlled nuclear fusion, 
Proceedings of the second United Nations international conference on the peaceful uses of atomic energy \tb{31},  
pp. 190-197 (1958).

\bibitem{Eisen} M. Eisenmberg and R. Guy,
\ti{A proof of the hairy ball
theorem}, Am. Math. Mon. \tb{86}, pp. 571-574 (1979).

\bibitem{Arnold} V. I. Arnold, 
\ti{On the topology of three-dimensional
steady flows of an ideal fluid}, J. Appl. Math. Mech. \tb{30}, pp. 223-226 (1966).

\bibitem{SalasI} D. Peralta-Salas \ti{Selected topics on the topology of ideal fluid flows}, International Journal of Geometric Methods in Modern Physics, \tb{13} (Supp. 1), 1630012 (2016).

\bibitem{Schwarz} G. Schwarz, in 
Hodge decomposition - a method for solving boundary value problems, Springer, pp. 67-72  (1995).

\bibitem{YosGiga} Z. Yoshida and Y. Giga, 
\ti{Remarks on spectra of operator rot}, 
Math. Z. \tb{204}, pp. 235-245 (1990). 

\bibitem{SalasBel} A. Enciso and D. Peralta-Salas, 
\ti{Beltrami fields with a nonconstant proportionality factor are rare}, 
Arch. Rat. Mech. Anal. \tb{220}, pp. 243-260 (2016).

\bibitem{SatoYamadaBel} N. Sato and M. Yamada, \ti{Local representation and construction of Beltrami fields}, Physica D: Nonlinear Phenomena \tb{391}, pp. 8-16 (2019).

\bibitem{Clelland} J. N. Clelland and T. Klotz, 
\ti{Beltrami fields with nonconstant proportionality factor}, 
\tb{236}, pp. 767-800 (2020). 

\bibitem{Abe} K. Abe, \ti{Rigidity of Beltrami fields with a non-constant proportionality factor}, J. Math. Phys.   \tb{63}, 041507 (2022). 

\bibitem{SalasBel2} D. Peralta-Salas and M. Vaquero, \ti{Beltrami Fields with Morse Proportionality Factor}, 	arXiv:2312.10511 (2023). 

\bibitem{Eden1} J. W. Edenstrasser, 
\ti{Unified treatment of symmetric MHD equilibria}, 
J. Plasma Phys. \tb{24}, pp. 299-313 (1980).

\bibitem{Eden2} J. W. Edenstrasser, 
\ti{The only three classes of symmetric
MHD equilibria}, J. Plasma Phys. \tb{24}, pp. 515-518
(1980).

\bibitem{Bruno} O. P. Bruno and P. Laurence,
\ti{Existence of three-dimensional toroidal MHD equilibria with nonconstant pressure}, 
Commun. Pure Appl. Math. \tb{49}, pp. 717-764 (1996).

\bibitem{Bob15} R. L. Dewar, Z. Yoshida, A.  Bhattacharjee, S. R. Hudson, \ti{Variational formulation of relaxed and multi-region relaxed magnetohydrodynamics}, Journal of Plasma Physics \tb{81} (6) 515810604 (2015).  

\bibitem{SalasMHS} A. Enciso, A. Luque, and D. Peralta-Salas, 
\ti{MHD Equilibria with nonconstant pressure in nondegenerate toroidal domains}, 
Journal of the European Mathematical Society, 
10.4171/JEMS/1410 (2023).

\bibitem{Grad} H. Grad,
\ti{Toroidal containment of a plasma},
The Physics of Fluids \tb{10}, 1 (1967).










\bibitem{Pasq} P. Constantin and F. Pasqualotto, 
\ti{Magnetic relaxation of a Voigt-MHD system}, Comm. Math. Phys. \tb{402}, pp.  1931-1952, (2023) 

 
\bibitem{Cao06} Y. Cao, E. M. Lunasin, and E. S. Titi, 
\ti{Global well-posedness of the three-dimensional viscous and inviscid simplified Bardina turbulence models}, 
Comm. Math. Sci. \tb{4}, 4, pp. 823-848 (2006).

\bibitem{Larios14} A. Larios and E. S. Titi, 
\ti{Higher-Order Global Regularity of an Inviscid Voigt-Regularization of the Three-Dimensional Inviscid Resistive Magnetohydrodynamic Equations}, 
J. Math. Fluid Mech. \tb{16}, pp. 59-76 (2014).

\bibitem{Larios10} A. Larios and E. S. Titi,
\ti{On the Higher-Order Global Regularity of the Inviscid Voigt-Regularization of Three-Dimensional Hydrodynamic Models}, 
Discrete and Continuous Dynamical Systems Series B \tb{14}, 2, pp. 603-627 (2010).

\bibitem{Ramos10} F. Ramos and E. S. Titi,
\ti{Invariant measures for the 3D Navier-Stokes-Voigt Equations and their Navier-Stokes limit}, Discrete and Continuous Dynamical Systems \tb{28}, 1 (2010).

\bibitem{Levant10} B. Levant, F. Ramos, and E. S. Titi, 
\ti{On the statistical properties of the 3D incompressible Navier-Stokes-Voigt model}, 
Commun. Math. Sci. \tb{8}, 1, pp. 277-293 (2010).














 




\bibitem{Grad67} H. Grad, 
\ti{The guiding center plasma}, 
Proc. Symp. Appl. Math. \tb{18}, pp. 162-248 (1967).

\bibitem{SatoYamada23} N. Sato and M. Yamada, 
\ti{Nested invariant tori foliating a vector field and its curl: toward MHD equilibria and steady Euler flows without continuous Eculidean isometries}, J. Math. Phys. \tb{64}, 081505 (2023).



\bibitem{Davids3} R. Cardona, 
N. Duignan, and D. Perrella, 
\ti{Asymmetry of MHD equilibria for generic adapted metrics}, 
arXiv:2312.14368v1 (2023).

\bibitem{Davids} D. Perrella, D. Pfefferl\'e, and L. Stoyanov, 
\ti{Rectifiability of divergence-free fields along invariant 2-tori}, 
Partial Differential Equations and Applications \tb{3}, 50 (2022).

\bibitem{Davids2} D. Perrella, 
N. Duignan, and D. Pfefferl\'e, 
\ti{Existence of global symmetries of divergence-free fields with first integrals}, 
J. Math. Phys. \tb{64}, 052705 (2023).

\bibitem{Constantin21} P. Constantin, T. D. Drivas, 
and D. Ginsberg, 
\ti{Flexibility and rigidity in steady fluid motion}, 
Comm. Math. Phys. \tb{385}, 
pp. 521-563 (2021).









\bibitem{YosClebsch} Z. Yoshida, \ti{Clebsch parametrization: basic properties and remarks on its applications}, 
J. Math. Phys. \tb{50}, 113101 (2009).

\bibitem{YosMor} Z. Yoshida and P. J. Morrison, \ti{Epi-two-dimensional fluid flow: A new topological paradigm for dimensionality}, Phys. Rev. Lett. \tb{119}, 244501 (2017).

\bibitem{GG} M. Golubitsky and V. Guillemin, \ti{Stable Mappings} in
Stable Mappings and Their Singularities, Springer-Verlag, New York, pp. 72-81 (1973).

\bibitem{SatoYamadaBoh} N. Sato and M. Yamada, 
\ti{A Reduced Ideal MHD System for Nonlinear Magnetic Field Turbulence in Plasmas with Approximate Flux Surfaces}, Journal of Mathematical Physics (to be published).























\end{thebibliography}
\end{document}